\documentclass[11pt]{article}
\usepackage[T1]{fontenc}
\usepackage{amsfonts}
\usepackage{amsmath}
\usepackage{amssymb}
\usepackage{amsthm}
\usepackage{bbm}
\usepackage{bm}
\usepackage{mathrsfs}
\usepackage{verbatim}
\usepackage{setspace}
\usepackage{color}
\usepackage{pdfsync}
\usepackage{enumitem}
\usepackage{graphicx}
\usepackage{subfigure}
\usepackage{tikz}
\usetikzlibrary{patterns}

\theoremstyle{plain}
\newtheorem{theorem}{Theorem}[section]
\newtheorem{proposition}[theorem]{Proposition}
\newtheorem{lemma}[theorem]{Lemma}

\theoremstyle{definition}
\newtheorem{definition}[theorem]{Definition}
\newtheorem{remark}[theorem]{Remark}

\newtheorem{example}[theorem]{Example}

\theoremstyle{remark}

\renewenvironment{thebibliography}[1]{%
\begin{oldthebibliography}{#1}%
\setlength{\baselineskip}{1em}
\linespread{.2}
\small
\setlength{\parskip}{0.25ex}%
\setlength{\itemsep}{.20em}%
}%
{%
\end{oldthebibliography}%
}
\newcommand{\eps}{\varepsilon}

\newcommand{\N}{\mathbb{N}}

\newcommand{\R}{\mathbb{R}}

\newcommand{\cF}{\mathcal{F}}

\newcommand{\cP}{\mathcal{P}}

\DeclareMathOperator{\diam}{diam}

\DeclareMathOperator{\id}{id}

\DeclareMathOperator{\Lip}{Lip}

\DeclareMathOperator*{\esssup}{ess\, sup}

\DeclareMathOperator*{\argmin}{arg\, min}

\newcommand{\as}{\mbox{-a.s.}}

\newcommand{\1}{\mathbf{1}}

\newcommand{\qforallq}{\quad\mbox{for all}\quad}
\newcommand{\qforq}{\quad\mbox{for}\quad}

\newcommand{\mykill}[1]{}

\numberwithin{equation}{section}

\usepackage[pdfborder={0 0 0}]{hyperref}
\hypersetup{
  urlcolor = black,
  pdfauthor = {Stephan Eckstein, Marcel Nutz},
  pdfkeywords = {Entropic Optimal Transport; Stability; Sinkhorn's Algorithm; IPFP},
  pdftitle = {Quantitative Stability of Regularized Optimal Transport and Convergence of Sinkhorn's Algorithm},
  pdfsubject = {Quantitative Stability of Regularized Optimal Transport and Convergence of Sinkhorn's Algorithm},
  pdfpagemode = UseNone
}

\begin{document}

\title{\vspace{-3em}
  Quantitative Stability of Regularized Optimal Transport and Convergence of Sinkhorn's Algorithm%
 }
\date{\today}

\author{
  Stephan Eckstein%
  \thanks{Department of Mathematics, ETH Zurich, seckstein@ethz.ch. Research supported by Landesforschungsf\"{o}rderung Hamburg under project LD-SODA. SE thanks Daniel Bartl, Mathias Beiglb\"{o}ck and Gudmund Pammer for fruitful discussions and helpful comments.
  }
  \and
  Marcel Nutz%
  \thanks{
  Departments of Statistics and Mathematics, Columbia University, mnutz@columbia.edu. Research supported by an Alfred P.\ Sloan Fellowship and NSF Grants DMS-1812661, DMS-2106056. MN is grateful to Guillaume Carlier, Giovanni Conforti, Flavien L\'eger and Luca Tamanini for their kind hospitality and advice.}
  }
\maketitle \vspace{-1.2em}

\begin{abstract}
We study the stability of entropically regularized optimal transport with respect to the marginals. Lipschitz continuity of the value and H\"older continuity of the optimal coupling in $p$-Wasserstein distance are obtained under general conditions including quadratic costs and unbounded marginals. The results for the value extend to regularization by an arbitrary divergence. As an application, we show convergence of Sinkhorn's algorithm in Wasserstein sense, including for quadratic cost. Two techniques are presented: The first compares an optimal coupling with its so-called shadow, a coupling induced on other marginals by an explicit construction. The second transforms one set of marginals by a change of coordinates and thus reduces the comparison of differing marginals to the comparison of differing cost functions under the same marginals.
\end{abstract}

\vspace{.3em}

{\small
\noindent \emph{Keywords} Entropic Optimal Transport; Stability; Sinkhorn's Algorithm; IPFP

\noindent \emph{AMS 2010 Subject Classification}
90C25; %
49N05 %
}
\vspace{.6em}

\section{Introduction}\label{se:intro}

Following advances allowing for computation in high dimensions, applications of optimal transport are thriving in areas such as machine learning, statistics, image and language processing (e.g., \cite{WGAN.17,ChernozhukovEtAl.17,RubnerTomasiGuibas.00,AlvarezJaakkola.18}). Regularization plays a key role in enabling efficient algorithms with provable convergence; see~\cite{PeyreCuturi.19} for a recent monograph with numerous references. Popularized in this context by~\cite{Cuturi.13}, entropic regularization is the most popular choice as it allows for Sinkhorn's algorithm (iterative proportional fitting procedure) that can be implemented at large scale using parallel computing and is analytically tractable. The entropically regularized transport problem can be formulated as
\begin{equation}\label{eq:defSentIntro}
  S_{\rm ent}^{\eps}(\mu_1, \mu_2, c) = \inf_{\pi \in \Pi(\mu_1,\mu_2)} \int c(x,y) \,\pi(dx,dy) + \eps D_{\rm KL}(\pi, \mu_1 \otimes \mu_2).
\end{equation}
Here $\Pi(\mu_1,\mu_2)$ is the set of couplings of the given marginals $\mu_1,\mu_2$ and $D_{\rm KL}(\cdot, \mu_1 \otimes \mu_2)$ is the Kullback--Leibler divergence relative to the product measure  $\mu_1 \otimes \mu_2$. Moreover, $\eps>0$ is a regularization parameter and $c$ is a cost function; the most important example is quadratic cost~$\|x-y\|^{2}$ on~$\R^{d}\times\R^{d}$. The basic idea is to solve \eqref{eq:defSentIntro} for small $\eps>0$ to obtain an approximation of the (unregularized) optimal transport problem that corresponds to $\eps=0$. Starting with~\cite{CominettiSanMartin.94,Mikami.02, Mikami.04} and followed by~\cite{ChenGeorgiouPavon.16,Leonard.12}, the convergence as $\eps\to0$ has been studied in detail and remains a very active area of investigation; see for instance \cite{AltschulerNilesWeedStromme.21, Berman.20, BerntonGhosalNutz.21, BlanchetJambulapatiKentSidford.18, ConfortiTamanini.19, GigliTamanini.21,NutzWiesel.21, Pal.19, Weed.18}.

The entropic optimal transport problem~\eqref{eq:defSentIntro} is also of its own interest. On the one hand, it is equivalent to a static formulation of the Schr\"odinger bridge problem that has a long history in physics (see \cite{Follmer.88, Leonard.14} for surveys); the dynamic Schr\"odinger bridge can be constructed by solving the static problem and combining it with a Brownian bridge. On the other hand, applied researchers have started to exploit numerous benefits resulting from entropic regularization, such as smoothness, existence of a gradient for gradient descent, improved sampling complexity (e.g., \cite{Corenflos2021differentiable,CuturiTeboulVert.19,genevay2019sample,GeneveyEtAl.16}), among many others.  Thus, the regularization is increasingly seen as an advantage rather than an approximation error; notions such as Sinkhorn divergence \cite{GenevayPeyreCuturi.18,RamdasGarciaTrillosCuturi.17} have become tools of their own right. We note that as long as $\eps>0$ is fixed, we can assume without loss of generality that $\eps=1$, simply dividing~\eqref{eq:defSentIntro} by $\eps$ and using the cost function $c/\eps$. Hence, we shall drop~$\eps$ from the formulation in our results.

The main objective of the present study is to establish and quantify the stability of the value~$S_{\rm ent}$ and its optimal coupling~$\pi^{*}$ with respect to the input marginals~$\mu_{1}$ and $\mu_{2}$, or more generally $\mu_1, \dots, \mu_N$ in the multi-marginal setting. Distances will be quantified by Wasserstein distance~$W_{p}$, thus allowing for comparison of measures with different supports, discrete and continuous measures, etc. We aim for results including unbounded marginals, replacing compactness by suitable integrability conditions such as the subgaussian tails in~\cite{MenaWeed.19}. Schr\"odinger bridges are one application where unbounded supports are very natural, as the Brownian dynamics produce unbounded intermediate marginals even if the boundary data are bounded. In this context, costs are usually quadratic, so that unbounded and non-Lipschitz cost functions are necessary. Even in applications with bounded costs, one may be interested in estimates with constants that do not depend on~$\|c\|_{\infty}$, especially not exponentially. 

To the best of our knowledge, the first stability result for entropic optimal transport is due to~\cite{CarlierLaborde.20}. Here, costs are uniformly bounded, and all marginals are equivalent to a common reference measure (e.g., Lebesgue), with densities uniformly bounded above and below. Within these families, distances of measures can be quantified by the $L^{p}$ norm of the difference of their densities. The authors show that the Schr\"odinger potentials (i.e., the dual entropic optimizers) are Lipschitz continuous relative to the marginals in $L^{p}$, for $p=2$ and $p=\infty$. This result is obtained by a differential approach establishing invertibility of the Schr\"odinger system.
More recently, \cite{GhosalNutzBernton.21b} obtain the first result on stability in a general setting. Using a geometric approach called cyclical invariance, continuity of optimizers is established in the sense of weak convergence. The geometric method avoids integrability conditions almost entirely and indeed remains valid even if the value of~\eqref{eq:defSentIntro} is infinite. On the other hand, the method relies on differentiation of measures which essentially forces the marginal spaces to be finite-dimensional. More importantly, the continuity result is purely qualitative, and that is the main difference with the present results. Most recently, and partly concurrently with the present study, a beautiful result of~\cite{DeligiannidisDeBortoliDoucet.21} establishes the uniform stability of Sinkhorn's  algorithm with respect to the marginals, in a bounded setting.  As a consequence, the authors deduce Lipschitzianity in~$W_{1}$ of the optimal couplings with respect to the marginals; the assumptions include bounded Lipschitz costs and bounded spaces. The argument is based on the Hilbert--Birkhoff projective metric which has also been used successfully to show linear convergence of Sinkhorn's algorithm \cite{ChenGeorgiouPavon.16b,FranklinLorenz.89}. A crucial additional step accomplished in~\cite{DeligiannidisDeBortoliDoucet.21}  is to pass from this metric to a more standard norm on the potentials. The techniques involving the projective metric are less probabilistic in nature, which may be one reason why it is wide open how to relax the boundedness conditions. We remark that the initial result of~\cite{CarlierLaborde.20} also covered the multimarginal problem which has recently become popular due to its role in the Wasserstein barycenter problem~\cite{AguehCarlier.11,CarlierEichingerKroshnin.20}. At least in the context of~\cite{Carlier.21}, it was observed that Hilbert--Birkhoff arguments may not be equally successful beyond two marginals. Finally, we mention the follow-up~\cite{NutzWiesel.22} on the continuity of the potentials in unbounded settings.

We apply our stability result to Sinkhorn's algorithm for $N=2$ marginals. It is well known that each iterate~$\pi^{n}$ of the algorithm solves an entropic optimal transport problem between its own marginals, and moreover these marginals converge to the given marginals~$\mu_{i}$. Thus, the convergence can be seen as a particular instance of stability with respect to marginals and our results apply. Sinkhorn's algorithm has been studied over almost a century (see~\cite{PeyreCuturi.19} for numerous references); the most general convergence results in this literature are due to~\cite{Ruschendorf.95}. While they treat costs that are merely measurable and show $\pi^{n}\to\pi^{*}$ in total variation, they do not cover unbounded functions like the quadratic cost in most examples, especially when both marginals have unbounded support. Applying stability results under regularity of~$c$ turns out to be fruitful in this regard: we not only  obtain the convergence to the optimal value  and $\pi^{n}\to\pi^{*}$ in Wasserstein distance, but even a rate of convergence. The conditions are sufficiently general to cover quadratic cost with subgaussian marginals.

\subsection{Synopsis}
Our first result, detailed in Theorem~\ref{th:valueConv}, is the continuity of the value $S_{\rm ent}$ with respect to the marginals in $p$-Wasserstein distance under generic conditions. If the cost $c$ is a product of suitably integrable Lipschitz functions, then $S_{\rm ent}$ is also Lipschitz. This includes quadratic costs on $\R^{d}$ with possibly unbounded marginal supports. The proof is based on comparing the optimizer~$\pi^{*}$ with the ``shadow'' coupling it induces on  other marginals. The shadow is a particular projection that we construct explicitly by gluing, controlling both the distance to $\pi^{*}$ and its divergence. The construction is simple and flexible, thus potentially useful for other purposes. For instance, Theorem~\ref{th:valueConv} holds for a general class of optimal transport problems regularized by a divergence~$D_{f}$ as previously considered in~\cite{DiMarinoGerolin.20b}, Kullback--Leibler divergence is a particular case. Other divergences, especially quadratic, are being used in some applications where entropic regularization performs poorly, usually because non-equivalent optimizers are desired or weak penalization (small $\eps$) causes numerical instabilities, see~\cite{blondel18quadratic,EssidSolomon.18,LorenzMannsMeyer.21}. Theoretical results are scarce so far as these regularizations are less tractable. 

By way of strong convexity, the continuity of the value~$S_{\rm ent}$ in Theorem~\ref{th:valueConv} leads to the continuity of the optimizer $\pi^{*}$ with respect to the marginals. Theorem~\ref{th:opti} states a nonasymptotic inequality bounding the distance of two entropic optimizers for different marginals in terms of the $W_{p}$ distance of the marginals. It shows in particular that the map $(\mu_{1},\dots,\mu_{N})\mapsto \pi^{*}$ is $1/(2p)$-H\"{o}lder in $W_{p}$. Exploiting a Pythagorean-type property of relative entropy to implement the strong convexity, we achieve an unbounded setting requiring only a transport inequality; i.e., a control of Wasserstein distance through entropy. This condition holds as soon as the marginals have a finite exponential moment; in particular, the result covers quadratic costs when marginals are $\sigma^{2}$-subgaussian for some (arbitrarily small)~$\sigma$. We remark that Theorem~\ref{th:valueConv} is the first quantitative stability result for unbounded costs, and in settings without differentiation of measures as assumed in~\cite{GhosalNutzBernton.21b}, even the qualitative result alone would be novel.

One noteworthy feature of Theorem~\ref{th:opti} is that the constants grow only linearly in $c$, which is particularly important for the regularized transport problem~\eqref{eq:defSentIntro}: here the effective cost function is $\tilde c:=c/\eps$ and $\eps$ is usually small. Many results on entropic optimal transport feature constants depending exponentially on the cost, typically $\exp(\|\tilde c\|_{\infty})$ or $\exp(\|\tilde c\|_{\infty} + \Lip \tilde c)$, including all previous results on stability that we are aware of. Even for well-behaved~$c$ on a fairly small domain, a choice like $\eps=.01$ then leads to constants far exceeding $e^{100}$, potentially a concern in practical considerations.

Our second continuity result, Theorem~\ref{th:second}, aims at improving the H\"older exponent in Theorem~\ref{th:opti} under the more restrictive condition that the cost~$c$ is bounded (spaces may still be unbounded). For instance, we show $1/(p+1)$-H\"{o}lder continuity in $W_{p}$. More generally, Theorem~\ref{th:second} yields the H\"older exponent $p/(p+1)q$ from $W_{p}$ to $W_{q}$; to wit, we can improve the exponent by measuring the distance of the marginals in a stronger norm. In particular, $p=\infty$ leads to a Lipschitz result into $W_{1}$. This choice also eliminates exponential dependence of the constant on the cost. In fact, we prove that the Lipschitz constant is \emph{sharp} in a nontrivial discrete example. This may be surprising given that the idea of proof is somewhat circuitous and that many estimates in this area are thought to be overly conservative.

Indeed, Theorem~\ref{th:second} is based on a novel approach that may be of independent interest; the basic idea is to reduce the problem of differing marginals to one of differing cost functions (under the same marginals). In the latter problem, optimizers are measure-theoretically equivalent and comparable in the sense of Kullback--Leibler divergence. Our starting point is the observation that the regularization in our problem depends only on the relative density, but not on the geometry of the distributions. In the simplest case, a $W_{p}$-optimal coupling of the differing marginals induces an invertible transport map~$T$ that can be used as change of coordinates to achieve identical marginals. The cost is transformed at the same time and we end up comparing~$c$ with~$c\circ T$. For this comparison, we can apply a separate result (Proposition~\ref{pr:costStability}) based on an entropy calculation.

The application to Sinkhorn's algorithm is summarized in Theorem~\ref{th:sinkhorn} which states convergence of the entropic cost and of the Sinkhorn iterates~$\pi^{n}$ themselves. The qualitative and quantitative results follow from Theorem~\ref{th:valueConv} and Theorem~\ref{th:opti}. In essence, the stability results turn a convergence rate for the Sinkhorn marginals into a convergence rate for~$\pi^{n}\to\pi^{*}$. We use the sublinear rate for the marginals as obtained in~\cite{Leger.21}. As noted there, these rates are likely suboptimal---for bounded cost functions, linear convergence of Sinkhorn's algorithm is well known~\cite{Carlier.21, ChenGeorgiouPavon.16b, FranklinLorenz.89}---our focus at this stage is on having \emph{some} quantitative control.

The organization of this paper is simple: Section~\ref{se:setting} details the setting, Section~\ref{se:results} presents the main results, and Section~\ref{se:proofs} contains  the proofs.

\section{Setting and Notation}\label{se:setting}

Let $(Y,d_{Y})$ be a Polish space and $\mathcal{P}(Y)$ its set of Borel probability measures. Given $p\in[1,\infty)$, we denote by $\mathcal{P}_p(Y)$ the subset of measures $\mu$ with finite $p$-th moment; i.e., $\int d_{Y}(x,\hat{x})^p \,\mu(dx) < \infty$ for some (and then all) $\hat{x} \in Y$. For $p=\infty$, we define  $\mathcal{P}_\infty(Y)$ as the measures with bounded support. The $p$-Wasserstein distance $W_p(\mu, \nu)$ between $\mu,\nu \in \mathcal{P}_p(Y)$ is defined via%
\begin{align*}
W_p(\mu, \nu)^p &= \inf_{\pi \in \Pi(\mu, \nu)} \int d_{Y}(x, y)^p \,\pi(dx, dy), \quad p\in [1,\infty),\\
W_\infty(\mu, \nu) &= \inf_{\pi \in \Pi(\mu, \nu)}\esssup_{(x, y) \sim \pi} d_{Y}(x, y),
\end{align*}
while $\|\mu - \nu\|_{TV} = \sup_{A\subseteq Y \,\text{Borel}} |\mu(A) - \nu(A)|$ is the total variation distance of $\mu, \nu \in \mathcal{P}(Y)$.

Fix $N\in \N$ and let $(X_i, d_{X_i})$, $i=1, \dots, N$ be Polish probability spaces with measures $\mu_{i}\in\cP(X_{i})$. We denote by $X = \prod_{i=1}^N X_i$ the product space and write $x\in X$ as $x = (x_1, \dots, x_N)$. When $p\in[1,\infty]$ is given, it will be convenient to use on~$X$ the particular product metric
\begin{align*}
  d_{X, p}(x, y):=
   \begin{cases}
   \big(\sum_{i=1}^N d_{X_i}(x_i, y_i)^p\big)^{1/p}, &\quad p\in[1,\infty),\\
   \max_{i=1,\dots,N} d_{X_i}(x_i, y_i), &\quad p=\infty.
\end{cases}
\end{align*}  
Unless otherwise noted, $p$-Wasserstein distances on~$X$ are understood with respect to $d_{X,p}$. Similarly, the distance between two tuples of marginals will often be quantified by
\begin{align*}
W_{p}(\mu_1, \dots, \mu_N;\tilde\mu_1, \dots, \tilde\mu_N) :=
\begin{cases}
\big(\sum_{i=1}^N W_p(\mu_i, \tilde{\mu}_i)^p\big)^{1/p}, &\quad p\in[1,\infty), \\
\max_{i=1,\dots,N} W_\infty(\mu_i, \tilde{\mu}_i), &\quad p=\infty.
\end{cases}
\end{align*} 
Given a Lipschitz function $c : X \rightarrow \mathbb{R}$, we denote by $\Lip_p(c)$ its Lipschitz constant with respect to $d_{X, p}$.

For a strictly convex, lower bounded function $f : \mathbb{R}_+ \rightarrow \mathbb{R}$ with $f(1) = 0$ and $\lim_{x \rightarrow \infty} f(x)/x = \infty$, the $f$-divergence $D_f(\mu,\nu)$ between probabilities $\mu,\nu$ on the same space is
\[
 D_f(\mu, \nu) := \int f\left(\frac{d\mu}{d\nu}\right) \,d\nu \qforq \mu\ll \nu
\]
and $D_f(\mu, \nu) :=\infty$ for $\mu\not\ll\nu$. The main example of interest to us is the Kullback--Leibler divergence (relative entropy) $D_{\rm KL}(\mu,\nu)$ which corresponds to the choice $f(x):=x\log x$. We always assume that $(\mu, \nu)\mapsto D_f(\mu, \nu)$ is lower semicontinuous for weak convergence. This holds for $D_{\rm KL}$, and more generally whenever $D_{f}$ has a suitable variational representation. %

Given $\mu_i \in \mathcal{P}(X_i)$ and a continuous, nonnegative\footnote{The lower bound is easily relaxed in view of the behavior of~\eqref{eq:defS} under shifts of~$c$.} cost function $c\in L^{1}(\mu_1 \otimes\dots \otimes \mu_N)$, we can now introduce the regularized transport problem 
\begin{equation}\label{eq:defS}
S(\mu_1, \dots, \mu_N, c) = \inf_{\pi \in \Pi(\mu_1, \dots, \mu_N)} \int c \,d\pi + D_f(\pi, \mu_1 \otimes \dots \otimes \mu_N),
\end{equation}
where $\Pi(\mu_1, \dots, \mu_N)\subset \cP(X)$ denotes the set of couplings of the marginals~$\mu_{i}$. Note that $S(\mu_1, \dots, \mu_N, c)<\infty$ by way of $\pi:=\mu_1 \otimes \dots \otimes \mu_N$. A standard argument of compactness and strict convexity then shows that~\eqref{eq:defS} admits a unique optimizer $\pi^{*}\in\Pi(\mu_1, \dots, \mu_N)$. When  $p\in[1,\infty)$ is given, we always assume that $c$ has growth of order $p$,
\begin{equation}\label{eq:growthcond}
  |c(x)|\leq C(1+d_{X, p}(x,\hat{x})^p)
\end{equation}
for some $C>0$ and  $\hat{x} \in X$, whereas for $p=\infty$ the meaning is that $c$ is bounded. For marginals $\mu_{i}\in \mathcal{P}_{p}(X_i)$, this ensures in particular that $c\in L^{1}(\pi)$ for any coupling~$\pi$.

While some of our results below hold for general divergences, we  use the notation $S_{\rm ent}$ in results specific to the entropic version, so that~\eqref{eq:defS} becomes 
\begin{equation}\label{eq:defSent}
  S_{\rm ent}(\mu_1, \dots, \mu_N, c) = \inf_{\pi \in \Pi(\mu_1, \dots, \mu_N)} \int c \,d\pi + D_{\rm KL}(\pi, \mu_1 \otimes \dots \otimes \mu_N).
\end{equation}

\begin{remark}\label{rk:otherRef}
  A variation of~\eqref{eq:defSent} uses entropy relative to a reference measure~$\hat P$ different from the product of the marginals,
\begin{equation}\label{eq:defSentOtherref}
  \inf_{\pi \in \Pi(\mu_1, \dots, \mu_N)} \int c \,d\pi + D_{\rm KL}(\pi, \hat P),
\end{equation}  
for instance (normalized) Lebesgue measure for problems with absolutely continuous marginals on $\R^{d}$. 
  Of course, a compatibility condition between~$\hat P$ and the marginals is necessary to guarantee that~\eqref{eq:defSentOtherref} is finite. As long as $\hat P=\hat P_{1}\otimes \cdots\otimes \hat P_{N}$ is a product measure, a standard computation shows that the optimizer~$\pi^{*}$ of this problem is the same as the one of~\eqref{eq:defSent}. Therefore, our stability results for~\eqref{eq:defSent} carry over to~\eqref{eq:defSentOtherref}.
\end{remark}

\section{Results}\label{se:results}

\subsection{Shadows and Preliminaries}\label{se:shadowsAndPrelims}

Given $\pi \in \Pi(\mu_1, \dots, \mu_N)$, we introduce a coupling $\tilde\pi \in \Pi(\tilde\mu_1, \dots, \tilde\mu_N)$ of different marginals through a gluing construction. Intuitively, for $N=2$, the transport $\tilde\pi$ is obtained by concatenating three transports: move $\tilde\mu_{1}$ to $\mu_{1}$ using a $W_{p}$-optimal transport, then follow the transport $\pi$ moving $\mu_{1}$ into $\mu_{2}$, and finally move $\mu_{2}$ to $\tilde\mu_{2}$ using a $W_{p}$-optimal transport. 
We think of~$\tilde\pi$ as a coupling of $\tilde\mu_1, \tilde\mu_2$  that ``shadows'' $\pi \in \Pi(\mu_1,\mu_2)$ as closely as possible given the differing marginals. The formal definition reads as follows. %

\begin{definition}[Shadow]\label{de:shadow}
  Let $p\in [1,\infty]$ and $\mu_i, \tilde{\mu}_i \in \mathcal{P}_p(X_i)$,  $i=1, \dots, N$. Let $\kappa_{i} \in \Pi(\mu_i, \tilde\mu_i)$ be a coupling attaining $W_{p}(\mu_i,\tilde\mu_i)$ and  $\kappa_{i}=\mu_i \otimes K_i$ a disintegration.
  Given $\pi \in \Pi(\mu_1, \dots, \mu_N)$, its \emph{shadow} $\tilde\pi \in \Pi(\tilde\mu_1, \dots, \tilde\mu_N)$ is defined as the second marginal of $\pi \otimes K\in\cP(X\times X)$, where the kernel $K : X \rightarrow \mathcal{P}(X)$ is defined as
	$%
	K(x) = K_1(x_1) \otimes \dots \otimes  K_N(x_N). 
	$%
\end{definition}

In general, the $W_{p}$-optimal kernel~$K_{i}$ need not be unique, so that there can in fact be more than one choice for the  shadow. Any choice will do in what follows, and we shall speak of ``the'' shadow despite the abuse of language. As detailed in Remark~\ref{rk:projection}, the shadow can also be understood as a particular choice of a $W_{p}$-projection of~$\pi$ onto~$\Pi(\tilde\mu_1, \dots, \tilde\mu_N)$. The crucial additional property of the shadow is that its divergence is controlled by the one of~$\pi$.

\begin{lemma}\label{le:shadow}
		Let $p\in [1,\infty]$ and $\mu_i, \tilde{\mu}_i \in \mathcal{P}_p(X_i)$,  $i=1, \dots, N$. Given $\pi \in \Pi(\mu_1, \dots, \mu_N)$, its shadow $\tilde{\pi} \in \Pi(\tilde{\mu}_1, \dots, \tilde{\mu}_N)$ satisfies
	\begin{align*}
	W_p(\pi, \tilde{\pi}) &= W_{p}(\mu_1, \dots, \mu_N;\tilde\mu_1, \dots, \tilde\mu_N),\\
	D_f(\tilde{\pi}, \tilde\mu_1 \otimes \dots \otimes \tilde\mu_N) &\leq D_f(\pi, \mu_1 \otimes \dots \otimes \mu_N).
	\end{align*}
\end{lemma}

To study the continuity properties of regularized optimal transport, we need to compare the cost of two couplings $\pi,\tilde\pi$ in the unregularized transport problem. If $c$ is $L$-Lipschitz, the following inequality holds for all probability measures $\pi,\tilde\pi$. We formulate an abstract condition to cover more general cases, especially Example~\ref{ex:quadraticccond} below.

\begin{definition}\label{de:ccond}
 Let $p\in [1,\infty]$ and $\mu_i, \tilde{\mu}_i \in \mathcal{P}_p(X_i)$,  $i=1, \dots, N$. For a constant $L\geq0$, we say that $c$ satisfies~\eqref{eq:ccond} if
  \begin{equation}\label{eq:ccond}
    \left| \int c \, d(\pi- \tilde\pi)\right| \leq  L W_{p}(\pi,\tilde\pi) \tag{${\rm A}_{L}$}
  \end{equation}
  for all $\pi \in \Pi(\mu_1, \dots, \mu_N)$ and $\tilde\pi \in \Pi(\tilde\mu_1, \dots, \tilde\mu_N)$.\footnote{In fact, \eqref{eq:ccond} will only ever be used when one coupling is the shadow of the other, but that restriction does not seem to substantially enhance the applicability.}
\end{definition}
The most important application is quadratic cost.

\begin{example}\label{ex:quadraticccond}
  For $p=2$ and cost $c(x_{1},x_{2})=\|x_{1}-x_{2}\|^{2}$ on Euclidean space $\R^{d}\times\R^{d}$,  we have that \eqref{eq:ccond} holds with
  $$
    L := \sqrt{2}\left[M(\mu_{1}) + M(\tilde\mu_{1}) + M(\mu_{2}) + M(\tilde\mu_{2})\right]
  $$
  where $M(\mu):=(\int \|x\|^{2}\,\mu(dx))^{1/2}$ for $\mu\in\cP(\R^{d})$.
\end{example}

The example is a special case of the following observation.

\begin{lemma}\label{le:ccond}
  Let $p\in [1,\infty)$. Let $c(x)=f(x)g(x)$ where $f,g$ are Lipschitz and have growth of order at most $p-1$. Then~\eqref{eq:ccond} holds with a constant~$L$ depending only on the Lipschitz and growth constants of $f,g$ and the $p$-th moments of $\mu_{i},\tilde\mu_{i}$, $i=1,\dots,N$. For $p=\infty$, the analogue holds with dependence on the bounds of $f,g$ instead of moments.
\end{lemma}

This criterion generalizes to a product $c(x)=c_{1}(x)\cdots c_{m}(x)$ of $m$ Lip\-schitz functions satisfying a suitable growth condition; cf.\ Remark~\ref{rk:ccondMultifactor}.

The next example shows that~\eqref{eq:ccond} also holds for the $p$-th power as cost.

\begin{example}\label{ex:pccond}
  For cost $c(x_{1},x_{2})=\|x_{1}-x_{2}\|^{p}$ with $p\in(1,\infty)$ on Euclidean space $\R^{d}\times\R^{d}$,  we have that \eqref{eq:ccond} holds with
  $$
    L := C_{p} \big[M_{p}(\mu_{1}) + M_{p}(\tilde\mu_{1}) + M_{p}(\mu_{2}) + M_{p}(\tilde\mu_{2})\big]^{p-1},
  $$
  where $M_{p}(\mu):=(\int \|x\|^{p}\,\mu(dx))^{1/p}$ for $\mu\in\cP(\R^{d})$ and $C_{p}$ is a constant depending only on~$p$.
\end{example}

 The proof, detailed in Section~\ref{se:proofs}, is similar to~\cite[Proposition~7.29]{Villani.09} and proceeds by estimating the derivative of a curve connecting the integrals in question. The example generalizes to costs $c(x_{1},x_{2})=\bar{c}(x_{1},x_{2})^{p}$ with~$\bar{c}$ being Lipschitz.
 
\subsection{Stability through Shadows}\label{se:stabilityThroughShadows}

We can now state our first result, establishing the continuity of~\eqref{eq:defS} with respect to the marginals. The qualitative part~(i) holds for general costs, the quantitative part~(ii) applies, in particular, to quadratic costs under 2-Wasserstein distance.

\begin{theorem}[Continuity of Value]\label{th:valueConv}
Let $p \in [1, \infty]$. 
\begin{itemize}
	\item[(i)] Let $\mu_i,\mu_i^n \in \mathcal{P}_p(X_i)$ satisfy $\lim_{n} W_p(\mu_i, \mu_i^n) = 0$ for $i=1, \dots, N$. 
	Then $S(\mu_1^n, \dots, \mu_N^n, c)\to S(\mu_1, \dots, \mu_N, c)$ and the associated optimal couplings converge in~$W_p$.
	
	\item[(ii)]
	Let $\mu_{i},\tilde{\mu}_i \in \mathcal{P}_p(X_i)$ for $i=1, \dots, N$ and let $c$ satisfy~\eqref{eq:ccond}. Then 
	\[
	\left|S(\mu_1, \dots, \mu_N, c) - S(\tilde\mu_1, \dots, \tilde\mu_N, c)\right| \leq L  W_{p}(\mu_1, \dots, \mu_N;\tilde\mu_1, \dots, \tilde\mu_N).
	\]
\end{itemize}
\end{theorem}

This result will be proved by comparing the cost of a coupling with the cost of its shadow. Using the same idea, we can show the convergence of the cost functionals as follows.

\begin{remark}[$\Gamma$-Convergence]\label{rk:gammaconv}
	  Define $\mathcal{F}: \mathcal{P}_p(X) \rightarrow \mathbb{R}\cup\{\infty\}$ by
	\begin{align*}
	\cF(\pi)= \begin{cases}
	\int c \,d\pi + D_f(\pi, \mu_1 \otimes \dots \otimes \mu_N) &\text{if } \pi \in \Pi(\mu_1, \dots, \mu_N),\\
	\infty, &\text{otherwise}
	\end{cases}
	\end{align*}
	and similarly $\cF_{n}$ for the marginals $\mu_{i}^{n}$. If $\lim_{n} W_p(\mu_i, \mu_i^n) = 0$, then $\cF_{n}$ $\Gamma$-converges to $\cF$ ; that is, given $\pi\in \mathcal{P}_p(X)$,
	\begin{itemize}
		\item[(a)] $\mathcal{F}(\pi) \leq \liminf \mathcal{F}_n(\pi_n)$ for any $(\pi_n)_{n\geq1}\subset\mathcal{P}_p(X)$ with $W_p(\pi, \pi_n) \to 0$,
		
		\item[(b)] there exists a  sequence $(\pi_n)_{n\geq1}\subset\mathcal{P}_p(X)$ with $W_p(\pi, \pi_n) \to 0$ and $\mathcal{F}(\pi) \geq \limsup \mathcal{F}_n(\pi_n).$
	\end{itemize}
	 For the recovery sequence in~(b), we can choose $\pi_{n}\in\Pi(\mu_1^n, \dots, \mu_N^n)$ to be the shadow of~$\pi\in\Pi(\mu_1, \dots, \mu_N)$.
\end{remark}
		
\begin{remark}\label{rk:varyc}
  Theorem \ref{th:valueConv}\,(i) and Remark~\ref{rk:gammaconv} generalize to a sequence of cost functions $c_{n}$ converging to~$c$ as long as the convergence is strong enough to imply $\int c_{n}\,d\pi_{n}\to \int c\,d\pi$ whenever $\pi_{n}\in\Pi(\mu_1^n, \dots, \mu_N^n)$ converge in~$W_{p}$ to some $\pi\in\Pi(\mu_1, \dots, \mu_N)$.
\end{remark}

Our second aim is to bound the distance between the optimizers for different marginals. The line of argument requires controlling Wasserstein distance through entropy, hence it is natural to postulate a transport inequality. Given $q \in [1, \infty)$, we say that $\mu_i\in\cP_{q}(X_{i})$, $i=1,\dots,N$ satisfy~\eqref{eq:transportineq} with constant $C_{q}$ if 
	\begin{equation}\label{eq:transportineq}
	W_q(\pi, \theta) \leq C_q D_{\rm KL}(\theta, \pi)^{\frac{1}{2q}} \qforallq \pi,\theta \in \Pi(\mu_1, \dots, \mu_N). \tag{${\rm I}_{q}$}
	\end{equation}
Similarly, they satisfy~\eqref{eq:transportineq2} with constant $C_{q}^{'}$ if
	\begin{equation}\label{eq:transportineq2}
	\begin{split} 
		W_q(\pi, \theta) \leq C_{q}^{'} \left[ D_{\rm KL}(\theta, \pi)^{\frac{1}{q}} + \left(\frac{D_{\rm KL}(\theta, \pi)}{2}\right)^{\frac{1}{2q}} \right] 
	\end{split} \tag{${\rm I}_{q}^{'}$}
	\end{equation} 
	for all $\pi,\theta \in \Pi(\mu_1, \dots, \mu_N)$. The two inequalities serve a similar purpose, but~\eqref{eq:transportineq2} is implied by a weaker integrability condition. Indeed, when $X$ is bounded, \eqref{eq:transportineq} holds as a simple consequence of Pinsker's inequality. Using the weighted inequalities of \cite{BolleyVillani.05}, \eqref{eq:transportineq} and \eqref{eq:transportineq2} also hold under much weaker exponential moment conditions on~$\mu_{i}$ as detailed in~(ii) and~(iii) below. In~(i), we obtain a different relaxation where all but one space~$X_{i}$ are bounded. Thus for the standard case~$N=2$, if one marginal is bounded, no condition at all is needed on the other marginal.

\begin{lemma}\phantomsection\label{le:transportineq}  
   \begin{enumerate}
		\item Let $X':=X_{2}\times\cdots\times X_{N}$ and suppose that
		\[
		  \diam_q(X') := \sup_{x, y \in X'} d_{X', q}(x, y) < \infty. \\[-.3em]
		\]
		Then~\eqref{eq:transportineq} holds with $C_q = 2^{-\frac{1}{2q}} \diam_q(X')$ for all  $\mu_i\in\cP_{q}(X_{i})$. 

		\item If $\mu_i\in\cP(X_{i})$ satisfy
		$
		\int \exp(\alpha \, d_{X_i}(\hat{x}_i, x_i)^{2q}) \, \mu_i(dx_i)< \infty
		$
		for some $\alpha \in (0, \infty)$ and $\hat x_{i}\in X_{i}$, then~\eqref{eq:transportineq} holds with constant
		\[
		C_q = 2 \inf_{\hat{x} \in X, \alpha > 0} \left(\frac{N}{2\alpha} \sum_{i=1}^N \bigg(1 + \log \int \exp(\alpha d_{X_i}(\hat{x}_i, x_i)^{2q}) \,\mu_i(dx_i) \bigg) \right)^{\frac{1}{2q}}.
		\]
		
		\item If $\mu_i\in\cP(X_{i})$ satisfy
		$
		\int \exp(\alpha \, d_{X_i}(\hat{x}_i, x_i)^{q}) \, \mu_i(dx_i)< \infty
		$
		for some $\alpha \in (0, \infty)$ and $\hat x_{i}\in X_{i}$, then~\eqref{eq:transportineq2} holds with constant
		\[
		C_q^{'} = 2 \inf_{\hat{x} \in X, \alpha > 0} \left(\frac{1}{\alpha} \sum_{i=1}^N \bigg( \frac{3}{2} + \log \int \exp(\alpha d_{X_i}(\hat{x}_i, x_i)^q) \mu_i(dx_i) \bigg) \right)^{\frac{1}{q}}.
		\]
	\end{enumerate}
\end{lemma}

Noting the logarithm in the formulas for~$C_{q}$ and $C_{q}^{'}$, we observe that these constants are typically much smaller than the exponential moment itself. We also note that the condition in~(iii) covers subgaussian marginals for~$q=2$.

We can now state a quantitative result for the stability of the optimizer of~\eqref{eq:defSent} relative to the marginals. In view of the above, the assumptions cover quadratic cost under 2-Wasserstein distance and subgaussian marginals.

\begin{theorem}[Stability of Optimizers]\label{th:opti}
	Let $p \in [1, \infty]$ and $q \in [1, \infty)$ with $q \leq p$, let $\mu_i,\tilde{\mu}_{i}\in \mathcal{P}_p(X_i)$, let $\mu_1,\dots,\mu_{N}$ satisfy~\eqref{eq:transportineq} with constant~$C_{q}$, and let $c$ satisfy~\eqref{eq:ccond}. Then the optimizers $\pi^*, \tilde{\pi}^*$ of $S_{\textrm{ent}}(\mu_1, \dots, \mu_N, c)$ and $S_{\textrm{ent}}(\tilde\mu_1, \dots, \tilde\mu_N, c)$ satisfy
	\[
	W_q(\pi^*, \tilde{\pi}^*) \leq N^{(\frac{1}{q} - \frac{1}{p})} \, \Delta +C_q \left(2L\, \Delta\right)^{\frac{1}{2q}}, \quad \Delta := W_{p}(\mu_1, \dots, \mu_N;\tilde\mu_1, \dots, \tilde\mu_N).
	\]
	If $\mu_1, \dots, \mu_N$ satisfy \eqref{eq:transportineq2} with constant $C_{q}^{'}$ instead of \eqref{eq:transportineq}, then
	\[
	W_q(\pi^*, \tilde{\pi}^*) \leq N^{(\frac{1}{q} - \frac{1}{p})} \, \Delta +C_q^{'} \left[ \left(2L\Delta\right)^{\frac{1}{q}} + \left(L\, \Delta\right)^{\frac{1}{2q}}\right].
	\]
	In particular, $(\mu_1, \dots, \mu_N)\mapsto \pi^*$ is $\frac{1}{2p}$-H\"{o}lder continuous in $W_p$ when restricted to a bounded set of marginals satisfying~\eqref{eq:ccond} and~$({\rm I}_{p})$ or~$({\rm T}'_{p})$ with given constants.
\end{theorem}

This result will be derived by comparing the optimizer with its shadow and applying a strong convexity argument, more specifically, a Pythagorean relation for relative entropy. 
In Theorem~\ref{th:opti}, only one set of marginals needs to satisfy~\eqref{eq:transportineq} or~\eqref{eq:transportineq2}. If the assumption holds for both $(\mu_i)$ and $(\tilde\mu_i)$, the proof shows that~$L$ can be replaced by~$L/2$ in the assertion.

\subsection{Stability through Transformation}\label{se:stabilityThroughTransform}

Next, we improve the H\"{o}lder exponent of Theorem~\ref{th:opti} for the case of bounded cost. The general line of argument is to reduce a difference in marginals to a difference in cost functions. Thus, we first state a stability result for the cost function under fixed marginals; it may be of independent interest.

\begin{proposition}[Stability wrt.\ Cost]\label{pr:costStability}
	Let $p \in [1, \infty]$, let $\mu_i\in \mathcal{P}_p(X_i)$, $i=1, \dots, N$ and $P = \mu_1 \otimes \dots \otimes \mu_N$. Let $c, \tilde{c}: X \rightarrow \mathbb{R}_{+}$ be bounded measurable, then the optimizers $\pi^*, \tilde{\pi}^*$ of $S_{\textrm{ent}}(\mu_1, \dots, \mu_N, c)$ and $S_{\textrm{ent}}(\mu_1, \dots, \mu_N, \tilde{c})$ satisfy
\begin{align*}
		\left\|\pi^* - \tilde{\pi}^*\right\|_{TV} 
		&\leq \frac{1}{2} a^{\frac{1}{p+1}} \, \|c-\tilde{c}\|_{L^p(P)}^{\frac{p}{p+1}}, \\
	  D_{\rm KL}(\pi^*, \tilde{\pi}^*) + D_{\rm KL}(\tilde{\pi}^*, \pi^*) 
	  &\leq a^{\frac{2}{p+1}} \|c-\tilde{c}\|_{L^p(P)}^{\frac{2p}{p+1}},
  	\end{align*}
		where $a :=  \exp(N\|c\|_\infty) +\exp(N\|\tilde{c}\|_\infty)$.
		Let $q\in [1,\infty)$. If $\mu_1,\dots,\mu_{N}$ satisfy~\eqref{eq:transportineq} with constant $C_q$, then also
		\[
		W_q(\pi^*, \tilde{\pi}^*) \leq 2^{-\frac{1}{2q}} C_q\left(a^\frac{1}{p} \|c-\tilde{c}\|_{L^p(P)}\right)^{\frac{p}{(p+1)q}}, 
		\]
		whereas if $\mu_1, \dots, \mu_N$ satisfy~\eqref{eq:transportineq2} with constant $C_q^{'}$, then
		\[
		W_q(\pi^*, \tilde{\pi}^*) \leq C_q^{'}\left[ \left(a^\frac{1}{p} \|c-\tilde{c}\|_{L^p(P)}\right)^{\frac{2p}{(p+1)q}} + 2^{-\frac{1}{2q}}\left(a^\frac{1}{p} \|c-\tilde{c}\|_{L^p(P)}\right)^{\frac{p}{(p+1)q}} \right].
		\]
\end{proposition}

(For $p=\infty$, the exponent $\frac{p}{(p+1)q}$ should be read as~$\frac{1}{q}$.) Proposition~\ref{pr:costStability} will be derived by comparing the optimizers in the sense of relative entropy $D_{\rm KL}(\pi^*, \tilde{\pi}^*)$. Of course, this is not possible in the other results where the marginals differ in a possibly singular way. We observe that the constant~$a$ deteriorates exponentially in $\|c\|_{\infty}$, however due to the $a^\frac{1}{p}$ in the formula this can be counteracted by using a stronger $L^{p}$ norm. In particular, for $p=\infty$, the direct dependence on $\|c\|_{\infty},\|\tilde c\|_{\infty}$ disappears completely, and moreover we obtain a Lipschitz estimate from $L^{\infty}$ to $W_{1}$.

Those features are inherited by our final result on the stability with respect to marginals; it improves the H\"{o}lder exponent of Theorem~\ref{th:opti} in the case of bounded  costs. As above, the dependence of the constant on~$\|c\|_{\infty}$ is avoided for $p=\infty$; we now obtain a Lipschitz result from $W_{\infty}$ into $W_{1}$.

\begin{theorem}[Stability of Optimizers for Bounded Cost]
	\label{th:second}
	Let $p \in [1, \infty]$ and $q \in [1, \infty)$ with $q \leq p$, let $\mu_i,\tilde{\mu}_{i}\in \mathcal{P}_p(X_i)$ satisfy~\eqref{eq:transportineq} with constant $C_{q}$ and let $c$ be bounded Lipschitz. Then the optimizers $\pi^*, \tilde{\pi}^*$ of $S_{\textrm{ent}}(\mu_1, \dots, \mu_N, c)$ and $S_{\textrm{ent}}(\tilde\mu_1, \dots, \tilde\mu_N, c)$ satisfy
		\[
		W_q(\pi^*, \tilde{\pi}^*) \leq N^{(\frac{1}{q} - \frac{1}{p})}\,\Delta +  2^{-\frac{1}{2q}}C_q  \left(a^{\frac{1}{p}} \,\Lip_p(c)\, \Delta\right)^{\frac{p}{(p+1)q}}
		\]
		where $a := 2 \exp(N\|c\|_\infty)$ and $\Delta := W_{p}(\mu_1, \dots, \mu_N;\tilde\mu_1, \dots, \tilde\mu_N)$. If $\mu_i, \tilde{\mu}_i$ satisfy \eqref{eq:transportineq2} with constant $C_q^{'}$ instead of \eqref{eq:transportineq}, then
		\begin{align*}
		W_q(\pi^*, \tilde{\pi}^*) &\leq N^{(\frac{1}{q} - \frac{1}{p})}\,\Delta \\&+  2^{-\frac{1}{q}} C_q^{'}\left[ \left(a^\frac{1}{p} \Lip_p(c)\, \Delta\right)^{\frac{2p}{(p+1)q}} + 2^{-\frac{1}{2q}}\left(a^\frac{1}{p} \Lip_p(c)\, \Delta\right)^{\frac{p}{(p+1)q}} \right].
		\end{align*}
		In particular, $(\mu_1, \dots, \mu_N)\mapsto \pi^*$ is $\frac{1}{p+1}$-H\"{o}lder continuous in $W_p$ when restricted to a bounded set of marginals satisfying~$({\rm I}_{p})$ or~$({\rm T}'_{p})$ %
	    with a given constant. For $q=1$ and $p=\infty$, we have the Lipschitz estimate 
    \[
		W_1(\pi^*, \tilde{\pi}^*) \leq \ell\, W_{\infty}(\mu_1, \dots, \mu_N;\tilde\mu_1, \dots, \tilde\mu_N)
		\]
		with constant $\ell:= N +  (C_{1}/\sqrt2) \Lip_{\infty}(c)$ independent of $\|c\|_{\infty}$. The constant~$\ell$ is sharp.
\end{theorem}

As discussed in the Introduction, this result is based on a transformation: instead of dealing with two sets of marginals, we use a change of coordinates to transform $\tilde{\mu}_{i}$ to $\mu_{i}$, at the expense of also transforming the cost function. For the resulting problem, we can apply Proposition~\ref{pr:costStability}. The sharpness of the constant~$\ell$ is discussed in Example~\ref{ex:sharpness}.

\begin{remark}\label{rk:generalMetrics}
	For simplicity, we have stated our results in the traditional setting where $W_{p}$ is defined through a metric compatible with the underlying Polish space. However, much of the above generalizes to any measurable metric. For instance, the discrete metric can be used to see that for $p=1$, our results include the total variation distance (see also~\cite{NutzWiesel.22} for further results on continuity in total variation). The majority of our arguments extend without change to the more general setting. In Definition~\ref{de:shadow}, it is no longer clear that there is a coupling attaining $W_{p}(\mu_{i},\tilde\mu_{i})$. However, we can use an $\epsilon$-optimal coupling to define an ``approximate shadow'' for which the first part of Lemma~\ref{le:shadow} is  replaced by $W_p(\pi, \tilde{\pi}) \leq  W_{p}(\mu_1, \dots, \mu_N;\tilde\mu_1, \dots, \tilde\mu_N) + \epsilon$, and then we can argue the main results as before.
	The extension to measurable metrics also applies to Proposition~\ref{pr:costStability}. Theorem~\ref{th:second} extends with the caveat that one needs to provide a substitute for the technical Lemma~\ref{lem:wloginvertible}\,(ii) in the specific metric under consideration, as its proof uses separability of the metric.
\end{remark}

\subsection{Application to Sinkhorn's Algorithm}

In this section we focus on $N=2$ marginals $\mu_{1},\mu_{2}$. Sinkhorn's algorithm is initialized at $\pi^{0} := P_c$, where $\frac{dP_c}{d(\mu_1 \otimes \mu_2)}(x) = \frac{\exp(-c(x))}{\int \exp(-c) \,d(\mu_1\otimes\mu_2)}$ is the Gibbs kernel associated with the cost~$c$. The Sinkhorn iterates $\pi^{n} \in \mathcal{P}(X)$, $n\geq1$ can then be defined recursively via
\begin{align*}
\frac{d\pi^{n}}{d\pi^{n-1}}(x) &:= \frac{d\mu_1}{d\pi^{n-1}_1}(x_1) && \hspace{-3em}\text{for $n$ odd,}\\
\frac{d\pi^{n}}{d\pi^{n-1}}(x) &:= \frac{d\mu_2}{d\pi^{n-1}_2}(x_2) &&\hspace{-3em}\text{for $n$ even,}
\end{align*}
where $\pi^{n-1}_i$ is the $i$-th marginal of~$\pi^{n-1}$. It follows that  
$\pi^{n}_{1}=\mu_{1}$ for~$n$ odd and $\pi^{n}_{2}=\mu_{2}$ for~$n$ even: for each iterate, one of the two marginals is the correct marginal. The other marginal does not match~$\mu_{i}$, but converges to it as $n\to\infty$. Importantly, each iterate~$\pi^{n}$ is the solution of an entropic optimal transport problem between its own marginals. As these marginals converge to~$(\mu_{1},\mu_{2})$, the convergence of Sinkhorn's algorithm can  be framed as a particular instance of stability with respect to the marginals. As above, we denote by~$\pi^*$ the optimizer of~$S_{\rm ent}(\mu_1, \mu_2, c)$. Moreover, we write $$\mathcal{F}(\pi) := \int c \,d\pi + D_{\rm KL}(\pi, \mu_1 \otimes \mu_2)$$ for the entropic cost of~$\pi \in \mathcal{P}(X)$, similarly as in Remark~\ref{rk:gammaconv} but without the penalty.

\begin{theorem}[Sinkhorn Convergence]\label{th:sinkhorn}
	Let $p \in [1, \infty)$. For $i=1,2$, let $\mu_i\in\cP(X_{i})$ satisfy
	$
	\int \exp(\alpha \, d_{X_i}(\hat{x}_i, x_i)^{p}) \, \mu_i(dx_i)< \infty
	$
	for some $\alpha \in (0, \infty)$ and $\hat x_{i}\in X_{i}$.
	\begin{itemize}
		\item[(i)] Let $c$ be continuous with growth of order~$p$. As $n\to\infty$, we have
		\begin{align*}
   		\cF(\pi^{n}) \to \cF(\pi^*), \quad \qquad \pi^{n} \to \pi^* \quad  \mbox{in} \quad W_{p}.
		\end{align*}

	\item[(ii)] Let $1\leq q\leq p$ and $c(x) = f(x) g(x)$ where $f, g$ are Lipschitz  with growth of order $p-1$. For all $n\geq2$, with a constant~$c_{0}$ detailed in the proof,
	\begin{align*}
		| \mathcal{F}(\pi^*) - \mathcal{F}(\pi^{n}) |  &\leq c_{0} n^{-\frac{1}{2p}}, \quad \qquad W_q(\pi^*, \pi^{n})\leq c_{0} n^{-\frac{1}{4pq}}.
	\end{align*}
	\end{itemize}
\end{theorem}

Theorem~\ref{th:sinkhorn} with $p=q=2$ implies $W_{2}$-convergence for quadratic cost with subgaussian marginals. 
 The form $c(x) = f(x) g(x)$ can be extended as in Remark~\ref{rk:ccondMultifactor}, or more generally to any condition guaranteeing~\eqref{eq:ccond} uniformly over the marginals produced by the algorithm. In particular, using Example~\ref{ex:pccond}, the assertion of the theorem also holds for $c(x)=\|x_{2}-x_{1}\|^{p}$.
The more detailed estimate given in the proof of the theorem shows that the constant~$c_{0}$ is at the same scale as~$c$; in particular, it does not grow exponentially with~$c$.

\section{Proofs}\label{se:proofs}

\subsection{Shadows and Preliminaries}\label{se:proofsShadowsAndPrelims}

For the convenience of the reader, we first recall the data processing inequality for our setting. Let $Y_1$ and $Y_2$ be Polish spaces. If $\mu\in\cP(Y_{1})$ and $K: Y_1 \rightarrow \mathcal{P}(Y_2)$ is a stochastic kernel,  we
\begin{equation}\label{eq:2ndMarginalNotation}
  \mbox{denote by $\mu K\in\cP(Y_{2})$ the second marginal of $\mu \otimes K\in\cP(Y_1 \times Y_2)$}.
\end{equation}

\begin{lemma}%
	\label{le:data}
	Let $\mu, \nu \in \mathcal{P}(Y_1)$ and $K: Y_1 \rightarrow \mathcal{P}(Y_2)$ a kernel. Then
	\[
	D_f(\mu K, \nu K) \leq D_f(\mu, \nu).
	\]
\end{lemma}	

\begin{proof}
  We may assume that $\mu\ll\nu$. For any kernels $K_1\ll K_2 : Y_1 \rightarrow \mathcal{P}(Y_2)$,
 \begin{equation}\label{eq:den} 
\frac{d(\mu \otimes K_1)}{d(\nu \otimes K_2)}(x, y) = \frac{d\mu}{d\nu}(x) \frac{dK_1(x)}{dK_2(x)}(y)  \quad \nu \otimes K_{2}\as
\end{equation}
  In particular,
  $
   \frac{d(\mu \otimes K)}{d(\nu \otimes K)}(x, y) = \frac{d\mu}{d\nu}(x)%
  $
  and thus
  \begin{equation}\label{eq:dpi1} 
   D_f(\mu, \nu) = D_f(\mu \otimes K, \nu \otimes K).
  \end{equation}
  Whereas in general, \eqref{eq:den} and Jensen's inequality for~$f$ yield
  \begin{align}\label{eq:dpi2} 
    D_f(\mu \otimes K_{1}, \nu \otimes K_{2}) &= \iint f\left( \frac{d\mu}{d\nu}(x) \frac{dK_1(x)}{dK_2(x)}(y) \right)\, K_{2}(x,dy)\nu(dx) \notag\\
    &\geq \int f\left( \frac{d\mu}{d\nu}(x)\right)\, \nu(dx) = D_f(\mu,\nu).
  \end{align} 
  Denote by
  $\mu \otimes K = (\mu K) \otimes \tilde{K}_1$ and $\nu \otimes K = (\nu K) \otimes \tilde{K}_2$  the ``reverse'' disintegrations from the second marginal to the first. Applying~\eqref{eq:dpi2} to $(\mu K) \otimes \tilde{K}_1$ and $(\nu K) \otimes \tilde{K}_2$,
  \begin{equation*} 
  D_f(\mu \otimes K, \nu \otimes K) = D_f((\mu K) \otimes \tilde{K}_1, (\nu K) \otimes \tilde{K}_2) \geq D_f(\mu K, \nu K).
  \end{equation*}
 In view of~\eqref{eq:dpi1}, this yields the claim.
\end{proof}

We can now show the two fundamental properties of the shadow.

\begin{proof}[Proof of Lemma~\ref{le:shadow}.]
	Let $\mu_i \otimes K_i \in \Pi(\mu_i, \tilde\mu_i)$ be a $W_p$-optimal coupling and define $\kappa = \pi \otimes K \in \mathcal{P}(X \times X)$ where 
	$
	K(x) = K_1(x_1) \otimes \dots \otimes  K_N(x_N),
	$
	so that $\tilde\pi := \pi K$ is the shadow of~$\pi$. In view of $\kappa \in \Pi(\pi, \tilde\pi)$, for $p < \infty$,
	\begin{align*}
	 W_p(\pi, \tilde\pi)^p &\leq \int d_{X, p}(x, y)^p \,\kappa(dx, dy) \\
	&= \int \sum_{i=1}^N d_{X_i}(x_i, y_i)^p \,\kappa(dx, dy) 
	= \sum_{i=1}^N  W_p(\mu_i, \tilde{\mu}_i)^p.
	\end{align*}
	On the other hand, given an  arbitrary coupling $\tilde\pi\in\Pi(\tilde\mu_1, \dots, \tilde\mu_N)$, any coupling $\gamma \in \Pi(\pi, \tilde{\pi})$ induces couplings $\gamma_{i} \in \Pi(\pi_{i}, \tilde{\pi}_{i})=\Pi(\mu_{i}, \tilde{\mu}_{i})$ of the individual marginals, hence
	\begin{align*}
	W_p(\pi, \tilde{\pi})^p &= \inf_{\gamma \in \Pi(\pi, \tilde{\pi})} \int \sum_{i=1}^N d_{X_i}(x_i, y_i)^p \,\gamma(dx, dy) \\
	&\geq \sum_{i=1}^N \inf_{\gamma_i \in \Pi(\mu_i, \tilde{\mu}_i)} \int d_{X_i}(x_i, y_i)^p \,\gamma_i(dx_i, d y_i) = \sum_{i=1}^N W_p(\mu_i, \tilde{\mu}_i)^p.
	\end{align*}
	The argument for $p=\infty$ is similar, completing the proof of the first claim. To show the bound on the divergence, note that 
	$\tilde\mu_1 \otimes \dots \otimes \tilde\mu_N = (\mu_1 \otimes \dots \otimes \mu_N)K$.
	Therefore, the data processing inequality (Lemma~\ref{le:data}) yields
	\[
	D_f(\tilde\pi, \tilde\mu_1 \otimes \dots \otimes \tilde\mu_N) = D_f(\pi K, (\mu_1 \otimes \dots \otimes \mu_N)K) \leq D_f(\pi, \mu_1 \otimes \dots \otimes \mu_N). \qedhere
	\]
\end{proof} 

\begin{remark}\label{rk:projection}
  The preceding proof shows that the shadow is a $W_{p}$-projection onto $\Pi(\tilde\mu_1, \dots, \tilde\mu_N)$; that is,
  $
    \tilde\pi \in \argmin_{\Pi(\tilde\mu_1, \dots, \tilde\mu_N)} W_{p}(\pi,\cdot).
  $ 
  In general, the argmin may have more than one element. A simple example on $\R\times\R$ is $\mu_1 = \mu_2 = \delta_0$ and $\tilde{\mu}_1 = \tilde{\mu}_2 = (\delta_{-1} + \delta_1)/2$; here any element of $\Pi(\tilde\mu_1,\tilde\mu_2)$ has the same distance to the singleton $\Pi(\mu_1,\mu_2)=\{\delta_{(0,0)}\}$. In this example, the shadow of $\pi:=\delta_{(0,0)}$ is unique. Clearly, not any projection is a shadow, and most projections fail to satisfy the divergence bound in Lemma~\ref{le:shadow}.
\end{remark}

Next, we show the criteria for~\eqref{eq:ccond}.

\begin{proof}[Proof of Lemma~\ref{le:ccond} and Example~\ref{ex:quadraticccond}.]
  To show the lemma, let $\kappa\in\Pi(\pi,\tilde\pi)$ be a coupling attaining $W_{p}(\pi,\tilde\pi)$. Then
  \begin{align}\label{eq:cproduct}
    &\int c \,  d(\pi- \tilde\pi) 
    = \int c(x) -c(y) \, \kappa(dx,dy) \notag\\
    & = \int f(x) (g(x)-g(y)) \, \kappa(dx,dy) + \int g(y) (f(x)-f(y)) \, \kappa(dx,dy).
  \end{align} 
  We estimate the first integral; the second is treated analogously. H\"older's inequality with $q$ such that $1/p+1/q=1$ yields
  \begin{align*}
    \int |f(x) (g(x)-g(y))| \, \kappa(dx,dy) \leq \|f\|_{L^{q}(\pi)} \|g(x)-g(y)\|_{L^{p}(\kappa)}.
  \end{align*}
  As $|f(x)|\leq C_{f} [1+d_{X_{1}}(x_{1},\bar x_{1})^{l} + \cdots+d_{X_{N}}(x_{N},\bar x_{N})^{l}]$ with $l\leq p-1=p(1-1/p)=p/q$ and hence $lq\leq p$, and as $\pi$ has marginals $\mu_{i}\in\cP_{p}(X_{i})$, we see that $\|f\|_{L^{q}(\pi)}$ is finite with a bound depending only on the $p$-th moments of $\mu_{i}$, $i=1,\dots,N$. On the other hand,
  $$
    \|g(x)-g(y)\|_{L^{p}(\kappa)} \leq \Lip_{p} (g) W_p(\pi, \tilde{\pi})
  $$
  due to the fact that $\kappa$ attains $W_p(\pi, \tilde{\pi})$. The lemma follows. 
  Example~\ref{ex:quadraticccond} follows from the above  estimate with $f(x)=g(x)=\|x_{1}-x_{2}\|$ in which case $\Lip_{2}(f)=\Lip_{2}(g)=\sqrt{2}$.
\end{proof} 

\begin{remark}\label{rk:ccondMultifactor}
  Lemma~\ref{le:ccond} can be generalized to a product of any finite number of Lipschitz functions.
  Let $c(x)=c_{1}(x)\cdots c_{m}(x)$ where $c_{j}$ are Lipschitz and decompose $c(x)-c(y)$ as in~\eqref{eq:cproduct} with $f(x):=c_{1}(x)\cdots c_{m-1}(x)$ and $g(x):=c_{m}(x)$. Proceeding inductively, we obtain that 
  $$
    c(x)-c(y) = \sum_{j=1}^{m} A_{j}(x,y) (c_{j}(x)-c_{j}(y))
  $$
  where $A_{j}(x,y)$ is a product of $m-1$ factors of the form $c_{k}(x)$ or $c_{l}(y)$. If $c_{j}(x)$, $j=1,\dots,m$ satisfy a growth condition suitably coordinated with a moment  condition on $\mu_{i},\tilde\mu_{i}$, then $\|A_{j}(x,y)\|_{L^{q}(\pi)}$ and $\|A_{j}(x,y)\|_{L^{q}(\tilde\pi)}$ can be bounded in terms of those moments and we deduce an analogue of Lemma~\ref{le:ccond}.
\end{remark}

\begin{proof}[Proof of Example~\ref{ex:pccond}.]
  Let $\kappa$ be a $W_p$-optimal coupling of $\pi$ and $\tilde{\pi}$. Set $\psi(x) := \|x\|^p$ and define $\varphi: [0, 1] \rightarrow \R$ by
	\[
	\varphi(t) := \int \psi((1-t) (x_2-x_1) + t (y_2-y_1)) \,\kappa(dx, dy);
	\]
	then $c(x)=\psi(x_2-x_1)$ and the quantity to be estimated is
	\begin{equation}\label{eq:pccondProof}
	  \left|\int c \,d\pi - \int c \,d\tilde\pi\right| = |\varphi(0) - \varphi(1)|.
	\end{equation}
	Using differentiation under the integral (justified by~\cite[Theorem~2.27]{Folland.99}), we see that $\varphi$ is differentiable and
	\[
	\frac{\partial \varphi}{\partial t}(t) = \int \big\langle \nabla \psi ((1-t) (x_2-x_1) + t (y_2-y_1)), (y_2-y_1-x_2+x_1)\big\rangle \,\kappa(dx, dy).
	\]
	Noting $\|\nabla \psi(v)\| = p\|v\|^{p-1}$ and writing $v_t = (1-t) (x_2-x_1) + t (y_2-y_1)$, the inequalities of Cauchy--Schwarz, H\"older and $(a+b)^{p}\leq 2^{p-1}(a^{p}+b^{p})$ yield
		\begin{align*}
	&\left| \frac{\partial \varphi}{\partial t}(t) \right| \leq \int \|\nabla \psi (v_t)\| \|(y_2-x_2) + (x_1-y_1)\| \,\kappa(dx, dy) \\
	&\leq \left(\int \|\nabla \psi  (v_t)\|^{\frac{p}{p-1}} \,\kappa(dx, dy)\right)^{\frac{p-1}{p}}  \!\left( \int \|(y_2-x_2) + (x_1-y_1)\|^{p} \,\kappa(dx, dy)\right)^{\frac{1}{p}} \\
	&\leq C'_p \left(\int \|v_t\|^{p} \,\kappa(dx, dy)\right)^{\frac{p-1}{p}} W_p(\pi, \tilde{\pi}) \\
	& \leq C_{p} \big[M_{p}(\mu_{1}) + M_{p}(\tilde\mu_{1}) + M_{p}(\mu_{2}) + M_{p}(\tilde\mu_{2})\big]^{p-1} W_p(\pi, \tilde{\pi})
	\end{align*}
	where $C_p,C_p'$ are constants depending only on $p$. In view of~\eqref{eq:pccondProof}, the claim follows.
\end{proof} 
\subsection{Stability through Shadows}\label{se:proofsStabilityThroughShadows}

We can now  show the continuity of the value.

\begin{proof}[Proof of Theorem \ref{th:valueConv}]
	(i) Let $\pi^{*},\pi^*_n$ be the optimizers for $S(\mu_1, \dots, \mu_N, c)$ and $S(\mu_1^n, \dots, \mu_N^n, c)$, respectively. For brevity, set $P = \mu_1 \otimes \dots \otimes \mu_N$ and $P_n = \mu_1^n \otimes \dots \otimes \mu_N^n$. 
	After passing to a subsequence, $\pi_{n}$ converges in $W_{p}$ to some $\pi\in\Pi(\mu_1, \dots, \mu_N)$, by weak compactness. We have 
	\[
	\liminf_{n\rightarrow \infty} \int c \,d \pi^*_n +D_f(\pi^*_n, P_n) \geq \int c \,d\pi + D_f(\pi, P) \geq \int c \,d\pi^* + D_f(\pi^*, P)
	\]
	by lower semicontinuity of $\int c\,d(\cdot)+D_f(\cdot, \cdot)$ and optimality of $\pi^*$.  On the other hand, let $\tilde{\pi}_n \in \Pi(\mu_1^n, \dots, \mu_N^n)$ be the shadow of $\pi^{*}$. Then Lemma~\ref{le:shadow} shows $\lim_{n} W_p(\tilde{\pi}_n, \pi^*) = 0$ and $D_f(\tilde{\pi}_n, P_n) \leq D_f(\pi^*, P)$, hence
	\begin{align*}
	\limsup_{n\rightarrow \infty} \int c \,d\pi^*_n + D_f(\pi^*_n, P_n) &\leq \limsup_{n\rightarrow \infty} \int c \,d\tilde{\pi}_n + D_f(\tilde{\pi}_n,  P_n)\\ &\leq \int c \,d\pi^* + D_f(\pi^*, P).
	\end{align*}
	Together, $\lim_n \int c \,d\pi^*_n + D_f(\pi^*_n, P_n) = \int c \,d\pi^* + D_f(\pi^*, P)$ and $\pi$ must be the (unique) optimizer $\pi^*$ of $S(\mu_1, \dots, \mu_N, c)$. In particular,  the original sequence $(\pi_{n}^*)$ converges to $\pi^{*}$, as claimed.
	
	(ii) Let $\pi^{*}$ be the optimizer of $S(\mu_1, \dots, \mu_N, c)$ and $\tilde{\pi} \in \Pi(\tilde{\mu}_1, \dots, \tilde{\mu}_N)$ its shadow. Using~\eqref{eq:ccond} and Lemma~\ref{le:shadow},
	\begin{align*}
	S(\mu_1, \dots, \mu_N, c) &= \int c\,d\pi^* + D_f(\pi^*, \mu_1 \otimes \dots \otimes \mu_N) \\
	&\geq \int c\,d\tilde{\pi} - L W_p(\pi^*, \tilde{\pi}) + D_f(\tilde{\pi}, \tilde{\mu}_1 \otimes \dots \otimes \tilde{\mu}_N) \\
	&\geq S(\tilde{\mu}_1, \dots, \tilde{\mu}_N, c) - L W_{p}(\mu_1, \dots, \mu_N;\tilde\mu_1, \dots, \tilde\mu_N).
\end{align*}	
The claim follows by symmetry.
\end{proof}

The proof of $\Gamma$-convergence follows the  same line of argument.

\begin{proof}[Proof of Remark~\ref{rk:gammaconv}]
  Similarly to the preceding proof, (a) follows from the lower semicontinuity of $\int c\,d(\cdot)+D_f(\cdot, \cdot)$. For~(b), let $\pi_{n}$ be the shadow of $\pi$ and use Lemma~\ref{le:shadow} to obtain $\int c \,d\pi_{n}\to\int c \,d\pi$ and $D_{f}(\pi_{n},\mu_1^n \otimes \dots \otimes \mu_N^n) \leq D_{f}(\pi,\mu_1 \otimes \dots \otimes \mu_N)$, again as in the preceding proof.
\end{proof}

The criteria for the transport inequality~\eqref{eq:transportineq} are derived as follows.

\begin{proof}[Proof of Lemma~\ref{le:transportineq}]
  (i) For the convenience of the reader, we first recall the standard argument for bounded~$X$: combine $d_{X, q}(x, y)^q \leq \diam_q(X)^q \, \1_{x \neq y}$ with the transport representation of total variation distance~\cite[Lemma~2.20]{Massart.07} and Pinsker's inequality~\cite[Theorem~2.16]{Massart.07} to obtain
	\begin{align*}
	  W_q(\pi, \theta)^{q} 
	  &= \inf_{\kappa\in\Pi(\pi,\theta)} \int d_{X,q}(x,y)^{q}\,\kappa(dx,dy) \\
	  &\leq   \diam_q(X)^q \inf_{\kappa\in\Pi(\pi,\theta)} \int\1_{x\neq y}\,\kappa(dx,dy) \\
	  &=  \diam_q(X)^q \|\pi - \theta\|_{TV} 
	  \leq \diam_q(X)^q \Big(\frac{1}{2} D_{\rm KL}(\theta, \pi)\Big)^{1/2}.
	\end{align*}
  The above holds for arbitrary probabilities~$\pi,\theta$. To prove the stronger estimate claimed in the lemma, we improve the above by exploiting that $\pi,\theta \in \Pi(\mu_1, \dots, \mu_N)$. Indeed, let $\Pi_{1}(\pi,\theta)\subset \Pi(\pi,\theta)$ denote the set of couplings~$\kappa\in\Pi(\pi,\theta)$ not moving mass in the $X_{1}$-direction; i.e.,
	$$
	\kappa \{(x_{1},\dots,x_{N},y_{1},\dots,y_{N}): x_{1}=y_{1}\}=1.
	$$
	Note that $\Pi_{1}(\pi,\theta)\neq\emptyset$ due to the fact that~$\pi$ and~$\theta$ have the same marginal~$\mu_{1}$ on~$X_{1}$. Clearly
		\begin{align*}
	  W_q(\pi, \theta)^{q} 
	  &= \inf_{\kappa\in\Pi(\pi,\theta)} \int d_{X,q}(x,y)^{q}\,\kappa(dx,dy) \\
	  &\leq \inf_{\kappa\in\Pi_{1}} \int d_{X,q}(x,y)^{q}\,\kappa(dx,dy) \\
	  &\leq   M^{q} \inf_{\kappa\in\Pi_{1}(\pi,\theta)} \int\1_{x\neq y}\,\kappa(dx,dy), \quad M:=\diam_{q}(X_{2}\times \cdots \times X_{N}).
	\end{align*}
  On the other hand, we claim that~$\pi,\theta$ having the same marginal implies
	\begin{align}\label{eq:TVfixedMarginal}
	  \inf_{\kappa\in\Pi_{1}(\pi,\theta)} \int\1_{x\neq y}\,\kappa(dx,dy) \leq \|\pi - \theta\|_{TV};
	\end{align}	
	in words, where mass needs to be moved, one might as well move only in the directions $X_{2},\dots,X_{N}$. Granted~\eqref{eq:TVfixedMarginal}, we can proceed as in the beginning and conclude the assertion of the lemma,
	\[
	  W_q(\pi, \theta)^{q} \leq M^{q} \|\pi - \theta\|_{TV} 
	  \leq M^{q} \Big(\frac{1}{2} D_{\rm KL}(\theta, \pi)\Big)^{1/2}.
	\]
	To show~\eqref{eq:TVfixedMarginal}, consider the mutually singular measures $\tilde\pi=\pi - (\pi\wedge\theta)$ and $\tilde\theta=\theta - (\pi\wedge\theta)$, where $\pi\wedge\theta$ is defined as usual via $d(\pi\wedge\theta)/d(\pi+\theta) = \min\{d\pi/d(\pi+\theta),d\theta/d(\pi+\theta)\}$. These measures again share a common first marginal, so that $\Pi_{1}(\tilde\pi,\tilde\theta)\neq\emptyset$. Let $\tilde{\kappa}\in \Pi_{1}(\tilde\pi,\tilde\theta)$ be arbitrary and let $\kappa\in\Pi(\pi,\theta)$ be the coupling given by $\kappa=\tilde{\kappa} + i$ where $i$ is the identical coupling of 
$\pi\wedge\theta$ with itself. Then
	\begin{align*}
	  \|\pi - \theta\|_{TV} \leq \int\1_{x\neq y}\,\kappa(dx,dy) 
	  = \int\1_{x\neq y}\,\tilde\kappa(dx,dy) =\|\tilde\pi-\tilde\theta\|_{TV}
	\end{align*} 
	where the last equality follows from mutual singularity. As $\|\tilde\pi-\tilde\theta\|_{TV}=\|\pi - \theta\|_{TV}$, all expressions are equal and~\eqref{eq:TVfixedMarginal} follows.
	
  (ii) It is shown in \cite[Corollary~2.4]{BolleyVillani.05} that the inequality \eqref{eq:transportineq} holds for a given measure $\pi\in\cP(X)$ and all $\theta\in\cP(X)$ whenever 
  \begin{equation}
		\label{eq:transportcond}
		\int \exp(\tilde\alpha \, d_{X, q}(x, \hat{x})^{2q}) \, \pi(dx) < \infty
		\end{equation}
  for some $\tilde\alpha>0$ and $\hat x\in X$, with constant
		\begin{equation}
		\label{eq:transportreference}
		C_{\pi, q} = 2 \inf_{\hat{x} \in X, \tilde\alpha > 0} \left( \frac{1}{2\tilde\alpha} \Big( 1 + \log \int \exp(\tilde\alpha d_{X, q}(\hat{x}, x)^{2q}) \pi(dx)\Big)\right)^{\frac{1}{2q}}.
		\end{equation}
		To obtain the claim for a coupling~$\pi$ (and general $\theta \in \mathcal{P}(X)$),  note that
		\[
		d_{X, q}(\hat{x}, x)^{2q} \leq N \sum_{i=1}^N d_{X, i}(\hat{x}_i, x_i)^{2q} = \frac{1}{N} \sum_{i=1}^N N^2 d_{X, i}(\hat{x}_i, x_i)^{2q}
		\] 
		and that the functional
		$
		f \mapsto \log \int \exp(\tilde\alpha f(x)) \,\pi(dx),
		$
		is convex (as can be seen from a variational representation,  e.g.\ \cite[Example~4.34, p.\,201]{FollmerSchied.11}). Hence
		\begin{align*}
		\log \!\int\! \exp(\tilde\alpha d_{X, q}(\hat{x}, x)^{2q}) \,\pi(dx) &\leq \frac{1}{N} \sum_{i=1}^N \log \!\int\! \exp(\tilde\alpha N^2 d_{X_i}(\hat{x}_i, x_i)^{2q}) \,\mu_i(d x_i).
		\end{align*}
		To obtain the claim for $C_{q}$, we plug this inequality into~\eqref{eq:transportreference} and set $\tilde\alpha=\alpha/N^{2}$. Similarly,  the integrability condition in the lemma implies~\eqref{eq:transportcond}.
	
	(iii) The proof is similar to~(ii) but refers to a different result of~\cite{BolleyVillani.05}. Indeed, by~\cite[Corollary~2.3]{BolleyVillani.05}, it suffices to bound
	\[
	C_{\pi, q}^{'} = 2 \inf_{\hat{x} \in X, \tilde\alpha > 0} \left( \frac{1}{\tilde\alpha} \Big( \frac{3}{2} + \log \int \exp(\tilde\alpha d_{X, q}(\hat{x}, x)^{q}) \pi(dx)\Big)\right)^{\frac{1}{q}}.
	\]
	Here the term inside the exponential already factorizes and we can directly apply the convexity of $f \mapsto \log \int \exp(\tilde{\alpha} f(x)) \pi(dx)$, which yields the claim after the substitution $\tilde{\alpha} = \alpha/N$.
\end{proof}

As a preparation for the proof of Theorem~\ref{th:opti}, we recall a Pythagorean relation for the entropic optimal transport problem. We denote
\[
  \cF(\pi) = \int c \,d\pi + D_{\rm KL}(\pi, \pi_1 \otimes \dots \otimes \pi_N)
\]
where $\pi_1, \dots, \pi_N$ are the marginals of $\pi$.  

\begin{lemma}
	\label{lem:Iprojection}
	If $\pi^*\in \Pi(\mu_1, \dots, \mu_N)$ is the optimizer of $S(\mu_1, \dots, \mu_N, c)$, 
	\[
	D_{\rm KL}(\pi, \pi^*) \leq \cF(\pi) - \cF(\pi^*) \qforallq \pi \in \Pi(\mu_1, \dots, \mu_N).
	\]
\end{lemma}
\begin{proof}
		Set $P=\mu_1 \otimes \dots \otimes \mu_N$ and define $P_{c}\in\cP(X)$ by 
		$d P_c = \alpha^{-1} e^{-c}\, dP$, where~$\alpha$ is the normalizing constant.
		Then
		\begin{equation}
		\label{eq:reformul}
		\cF(\pi) = D_{\rm KL}(\pi, P_c) - \log \alpha,
		\end{equation}
		so that the entropic optimal transport problem~\eqref{eq:defSent} is equivalent to minimizing $D_{\rm KL}(\cdot, P_{c})$. In particular,  $\pi^{*}= \argmin_{\Pi(\mu_1, \dots, \mu_N)} D_{\rm KL}(\cdot, P_{c})$ and the Pythagorean theorem for relative entropy \cite[Theorem~2.2]{Csiszar.75} yields
		\[
		D_{\rm KL}(\pi, P_c) \geq D_{\rm KL}(\pi^*, P_c) + D_{\rm KL}(\pi, \pi^*)  \qforallq \pi \in \Pi(\mu_1, \dots, \mu_N).
		\]
		In view of~\eqref{eq:reformul}, the claim follows.	(In the case under consideration, the assertion holds even with equality. We do not need that fact here.)
\end{proof}

We can now show the stability of optimizers with respect to the marginals.

\begin{proof}[Proof of Theorem \ref{th:opti}]
  We detail the proof for \eqref{eq:transportineq}; the argument for~\eqref{eq:transportineq2} is identical. For notational convenience, we treat the case where $\tilde{\mu}_i$ (rather than $\mu_{i}$) satisfy~\eqref{eq:transportineq}. Consider the optimizers $\pi^{*}\in\Pi(\mu_1, \dots, \mu_N)$ and $\tilde{\pi}^{*} \in \Pi(\tilde{\mu}_1, \dots, \tilde{\mu}_N)$. 
  Let $\tilde{\pi} \in \Pi(\tilde{\mu}_1, \dots, \tilde{\mu}_N)$ be the shadow of $\pi^{*}$ for the $p$-Wasserstein distance. Using Lemma~\ref{le:shadow} and~\eqref{eq:ccond} as in the proof of Theorem~\ref{th:valueConv}\,(ii),
  \[
	\cF(\tilde{\pi}) - \cF(\pi^*)  \leq \int c\,d(\tilde{\pi}-\pi^*) \leq L \, W_p(\tilde{\pi}, \pi^*) \leq L \Delta.
	\]
	We also have $\cF(\pi^*) - \cF(\tilde{\pi}^{*}) \leq L \Delta$ by Theorem~\ref{th:valueConv}\,(ii), and adding the inequalities yields
  \[
    \cF(\tilde\pi) - \cF(\tilde{\pi}^*) \leq 2L \Delta.
  \]
  (If both marginals satisfy~\eqref{eq:transportineq} with constant~$L$, we can assume by symmetry that $\cF(\pi^*) - \cF(\tilde{\pi}^{*}) \leq 0$, and obtain the estimate with $L$ instead of $2L$.)
  By Lemma~\ref{lem:Iprojection}, it follows that
	$
	D_{\rm KL}(\tilde{\pi}, \tilde{\pi}^*) \leq 2L \Delta,
	$
	and now~\eqref{eq:transportineq} implies
	\[
	W_q(\tilde{\pi}, \tilde{\pi}^*) \leq C_q (2L \Delta)^\frac{1}{2q}.
	\]
	Recalling that~$W_{r}$ on~$X$ was defined relative to the distance $d_{X, r}$, Jensen's inequality implies $W_q(\cdot,\cdot) \leq N^{(\frac{1}{q} - \frac{1}{p})} W_p(\cdot,\cdot)$. In view of Lemma~\ref{le:shadow}, we deduce $W_q(\pi^*, \tilde{\pi}) \leq N^{(\frac{1}{q} - \frac{1}{p})} W_p(\pi^*, \tilde{\pi}) \leq N^{(\frac{1}{q} - \frac{1}{p})} \Delta$. We conclude the proof via the triangle inequality,
	\[
	W_q(\pi^*, \tilde{\pi}^*) \leq W_q(\pi^*, \tilde{\pi}) + W_q(\tilde{\pi}, \tilde{\pi}^*) \leq N^{(\frac{1}{q} - \frac{1}{p})} \Delta + C_q \left(2L \Delta\right)^{\frac{1}{2q}}.\qedhere
	\]
\end{proof}

\subsection{Stability with respect to Cost}

Throughout this section, we fix $p \in [1, \infty]$, $\mu_i \in \mathcal{P}_p(X_i)$ for $i=1, \dots, N$ and $c, \tilde{c}: X \rightarrow [0,\infty)$ satisfying the growth condition~\eqref{eq:growthcond}.
The following observation is the starting point for the stability with respect to the cost function. We recall that $P := \mu_1 \otimes \dots \otimes \mu_N$.

\begin{lemma}
	\label{le:coolinequality}
	Let $\pi^*, \tilde{\pi}^*$ be the respective optimizers of $S_{\rm ent}(\mu_1, \dots, \mu_N, c)$ and $S_{\rm ent}(\mu_1, \dots, \mu_N, \tilde{c})$. Then
	  \[
		D_{\rm KL}(\pi^*, \tilde{\pi}^*) + D_{\rm KL}(\tilde{\pi}^*, \pi^*) \leq \int (c - \tilde{c}) \, d(\tilde{\pi}^* - \pi^*).
		\]
	\begin{proof}
		Lemma~\ref{lem:Iprojection} yields
		\begin{align*}
		D_{\rm KL}(\pi^*, \tilde{\pi}^*) + D_{\rm KL}(\tilde{\pi}^*, \pi^*) 
		&\leq \int c \,d\tilde{\pi}^* + D_{\rm KL}(\tilde{\pi}^*, P) + \int \tilde{c} \,d\pi^* + D_{\rm KL}(\pi^*, P)\\
		&- \int c \,d\pi^* - D_{\rm KL}(\pi^*, P) - \int \tilde{c} \,d\tilde{\pi}^* - D_{\rm KL}(\tilde{\pi}^*, P)\\
		&= \int (c-\tilde{c}) \,d(\tilde{\pi}^* - \pi^*).
		\qedhere
		\end{align*}
	\end{proof}
\end{lemma}

Lemma~\ref{le:coolinequality} clearly implies a Lipschitz estimate with respect to $\|c - \tilde{c}\|_{\infty}$ by using Pinsker's inequality on the left-hand side. The following proof is a variation on that observation.

\begin{proof}[Proof of Proposition~\ref{pr:costStability}]
  Combining 
	\[
		\int (\tilde{c} - c) \,d(\pi^* - \tilde{\pi}^*) \leq \int |\tilde{c} - c| \left| \frac{d\pi^*}{dP} - \frac{d\tilde\pi^*}{dP}\right| dP
		\]
		with H\"{o}lder's inequality as well as (in case $p \neq 1$), for $q := \frac{p}{p-1}$,
		\[
		\left| \frac{d\pi^*}{dP} - \frac{d\tilde\pi^*}{dP}\right|^q \leq \left\| \frac{d\pi^*}{dP} - \frac{d\tilde\pi^*}{dP}\right\|_{L^\infty(P)}^{q-1} \left| \frac{d\pi^*}{dP} - \frac{d\tilde\pi^*}{dP}\right| ,
		\]
		yields 
  \begin{equation}\label{eq:costStability1}
  \int (\tilde{c} - c) \,d(\pi^* - \tilde{\pi}^*) \leq  \|\tilde{c}-c\|_{L^p(P)} (2\|\pi^* - \tilde{\pi}^*\|_{TV})^{1-\frac{1}{p}} \left\|\frac{d\pi^*}{dP} - \frac{d\tilde{\pi}^*}{dP}\right\|_{\infty}^\frac{1}{p}.
	\end{equation}
  Next, we show 
  \begin{equation}\label{eq:costStability2}
	\left\|\frac{d\pi^*}{dP} - \frac{d\tilde{\pi}^*}{dP}\right\|_{\infty} \leq a := \exp(N\|c\|_\infty) +\exp(N\|\tilde{c}\|_\infty).
  \end{equation}
  Recall that by duality (e.g., \cite{DiMarinoGerolin.20,Nutz.20}), for certain ``potentials'' $\varphi_i : X_i \rightarrow \mathbb{R}$,
		\begin{equation}\label{eq:duality1}
		\frac{d\pi^*}{dP}(x) = \exp\left(-c + \oplus_{i}\varphi_{i}\right)
  \end{equation}
		where $(\oplus_{i} \varphi_i)(x) := \sum_{i=1}^{N} \varphi(x_{i})$, and moreover
		\begin{equation}\label{eq:duality2}
		  \int \oplus_{i} \varphi_i \,dP = S_{\rm ent}(\mu_1, \dots, \mu_N, c) \geq 0
		\end{equation}
		where the inequality is due to $c\geq0$. To estimate the right-hand side of~\eqref{eq:duality1}, recall that~\eqref{eq:duality1} and the fact that $\pi^{*}$ is a coupling imply a conjugacy relation between the potentials (e.g., \cite{DiMarinoGerolin.20,Nutz.20,NutzWiesel.21}), namely
		\begin{align*}
		\varphi_i(x_i) &= - \log \int \exp (-c(x) + \oplus_{j \neq i} \varphi_j(x_{j}) ) \,P_{-i}(d x_{-i}) \\
		&\leq \|c\|_{\infty} - \int \oplus_{j \neq i} \varphi_j \,dP_{-i},
		\end{align*}		
		where $x_{-i}:=(x_{1},\dots,x_{i-1},x_{i+1},\dots,x_{N})$ and 
		$P_{-i}:=\otimes_{j\neq i} \mu_{j}$. Thus by~\eqref{eq:duality2}, 
		\[
		\oplus_{i} \varphi_i (x) \leq N \|c\|_\infty - (N-1) \int \oplus_{j = 1}^N \varphi_j \,dP \leq N \|c\|_\infty.
		\]
		Using this in~\eqref{eq:duality1}, we conclude that
		\[
		  \left\|\frac{d\pi^*}{dP}\right\|_{\infty} \leq \exp (N \|c\|_\infty).
		\]
		The analogue holds for $\tilde{\pi}^*$, hence $\left\|\frac{d\pi^*}{dP} - \frac{d\tilde{\pi}^*}{dP}\right\|_{\infty} \leq \left\|\frac{d\pi^*}{dP}\right\|_{\infty} + \left\|\frac{d\tilde{\pi}^*}{dP}\right\|_{\infty} \leq a$ as claimed in~\eqref{eq:costStability2}.

	  Pinsker's inequality, Lemma~\ref{le:coolinequality}, \eqref{eq:costStability1} and~\eqref{eq:costStability2} imply
		\begin{align*}
		4 \|\pi^* - \tilde{\pi}^*\|^2_{TV} &\leq D_{\rm KL}(\pi^*, \tilde{\pi}^*) + D_{\rm KL}(\tilde{\pi}^*, \pi^*)\\ &\leq \int (c - \tilde{c}) \, d(\tilde{\pi}^* - \pi^*) 
		\leq  a^{\frac{1}{p}} (2\|\pi^* - \tilde{\pi}^*\|_{TV})^{1-\frac{1}{p}} \|\tilde{c} - c\|_{L^p(P)}.
		\end{align*}
		Dividing by $4 \|\pi^* - \tilde{\pi}^*\|_{TV}^{1-\frac{1}{p}}$ yields
	\begin{equation}\label{eq:costStability3}
	  \|\pi^* - \tilde{\pi}^*\|_{TV}^{1+\frac{1}{p}} \leq \Big(\frac{1}{2}\Big)^{1+\frac{1}{p}} a^{\frac{1}{p}} \|\tilde{c}-c\|_{L^p(P)}
	\end{equation}
	which is the first claim of the proposition. On the other hand, using Lemma~\ref{le:coolinequality} and \eqref{eq:costStability1} together with \eqref{eq:costStability3} yields  
	\begin{equation}\label{eq:duality4}	  
	D_{\rm KL}(\pi^*, \tilde{\pi}^*) + D_{\rm KL}(\tilde{\pi}^*, \pi^*) \leq a^{\frac{1}{p}} \|\tilde{c}-c\|_{L^p(P)} \left(a^{\frac{1}{p}} \|\tilde{c}-c\|_{L^p(P)}\right)^{\frac{p-1}{p+1}}.
  \end{equation}
  As \eqref{eq:transportineq} implies
	$
	2 C_q^{-2q}W_q(\pi^*, \tilde{\pi}^*)^{2q} \leq D_{\rm KL}(\pi^*, \tilde{\pi}^*) + D_{\rm KL}(\tilde{\pi}^*, \pi^*),
	$
	this proves the second claim of the proposition. For the last claim, we drop the nonnegative term $D_{\rm KL}(\tilde\pi^*, \pi^*)$ on the left-hand side of~\eqref{eq:duality4} and use~\eqref{eq:transportineq2} with the remaining inequality.
\end{proof}

\subsection{Stability through Transformation}\label{se:proofsStabilityThroughTransform}

Let $p \in [1, \infty]$, $\mu_i,\tilde\mu_{i} \in \mathcal{P}_p(X_i)$ for $i=1, \dots, N$ and let $c: X \rightarrow [0,\infty)$ satisfy the growth condition~\eqref{eq:growthcond}.
We begin with preliminary results connecting stability with respect to the marginals and stability with respect to the cost function. As in Definition~\ref{de:shadow}, $K$~denotes the kernel $K(x) = K_1(x_1) \otimes \dots \otimes K_N(x_N)$, where $\mu_i \otimes K_i \in \Pi(\mu_i, \tilde{\mu}_i)$ is an optimal coupling attaining $W_p(\mu_i, \tilde{\mu}_i)$. We use the notation $Kc(x) := \int c(y) \,K(x, dy)$ for the integral of~$c$ with respect to the kernel.

\begin{lemma}
	\label{lem:costislp}
	Let $p\in [1,\infty]$ and let $c$ be $\Lip_p(c)$-Lipschitz. Then
	\begin{align*}
	\|c - Kc \|_{L^p(\pi)} \leq \Lip_p(c) W_{p}(\mu_1, \dots, \mu_N;\tilde\mu_1, \dots, \tilde\mu_N), \quad \pi \in \Pi(\mu_1, \dots, \mu_N).
	\end{align*}
\end{lemma}

\begin{proof}
  We only detail the calculation for $p<\infty$,
  \begin{align*}
    \|c - Kc \|_{L^p(\pi)}^p &= \int \Big|c(x) - \int c(y) K(x, dy)\Big|^p \pi(dx) \\
		&\leq \iint |c(x) - c(y)|^p K(x, dy) \pi(dx) \\
		&\leq \Lip_p(c)^p \iint \sum_{i=1}^N d_{X, i}(x_i, y_i)^p K(x, dy) \pi(dx) \\
		&= \Lip_p(c)^p \sum_{i=1}^N W_p(\mu_i, \tilde{\mu}_i)^p. \qedhere
	\end{align*}
\end{proof}

Next, consider the kernel $\tilde K$ defined like $K$ but with the marginals reversed; that is, $\tilde K(x) = \tilde K_1(x_1) \otimes \dots \otimes \tilde K_N(x_N)$, where $\tilde\mu_i \otimes \tilde K_i \in \Pi(\tilde\mu_i, \mu_i)$ is an optimal coupling attaining $W_p(\tilde{\mu}_i,\mu_i)$.  The double integral $\tilde{K}Kc :=\tilde{K} (Kc) $ thus corresponds to a round-trip between the marginals. In general, this round-trip leads to a positive gap~$R$ in value, as shown in the next result. The result will not be used in the subsequent proofs but it may be useful to understand the steps below, where we look for situations where the gap is zero.

\begin{lemma}
	\label{lem:prepcost}
	Let $p\in [1,\infty]$. We have
	\[
	S(\tilde{\mu}_1, \dots, \tilde{\mu}_N, c) \leq S(\mu_1, \dots, \mu_N, Kc) \leq S(\tilde{\mu}_1, \dots, \tilde{\mu}_N, c) + R,
	\]
	where
	$R:= \int (\tilde{K}Kc  - c) \,d\tilde{\pi}^*$ and $\tilde{\pi}^*$ is the optimizer of $S(\tilde{\mu}_1, \dots, \tilde{\mu}_N, c)$. Moreover, $R\leq 2 \Lip_p(c) W_{p}(\mu_1, \dots, \mu_N;\tilde\mu_1, \dots, \tilde\mu_N)$.
\end{lemma}

\begin{proof}
		Set $\tilde{P}=\tilde{\mu}_1\otimes \cdots\otimes \tilde{\mu}_N$ and recall~\eqref{eq:2ndMarginalNotation}. Using Lemma~\ref{le:data} twice,
		\begin{align*}
		S(\tilde{\mu}_1, \dots, \tilde{\mu}_N, c) &= \inf_{\tilde{\pi} \in \Pi(\tilde{\mu}_1, \dots, \tilde{\mu}_N)} \int c \,d\tilde{\pi} + D_f(\tilde{\pi}, \tilde{P}) \\
		&\leq \inf_{\pi \in \Pi(\mu_1, \dots, \mu_N)} \int c \,d(\pi K) + D_f(\pi K, P K) \\
		&\leq \inf_{\pi \in \Pi(\mu_1, \dots, \mu_N)} \int Kc \,d\pi + D_f(\pi, P) \\
		&= S(\mu_1, \dots, \mu_N, Kc) \\
		&\leq \int Kc \,d(\tilde{\pi}^* \tilde{K}) + D_f(\tilde{\pi}^* \tilde{K}, \tilde{P} \tilde{K}) \\
		&\leq \int \tilde{K}Kc \, d\tilde{\pi}^* + D_f(\tilde{\pi}^*, \tilde{P}) 
		= S(\tilde{\mu}_1, \dots, \tilde{\mu}_N, c) + R.
		\end{align*}
		The bound for $R$ is similar to the proof of Lemma~\ref{lem:costislp}.
\end{proof}

In Lemma~\ref{lem:prepcost}, there is a gap between the values of $S(\tilde{\mu}_1, \dots, \tilde{\mu}_N, c)$ and $S(\mu_1, \dots, \mu_N, Kc)$. If however the kernels $K,\tilde K$ are given by maps inverse to one another (as will be the case in the proof of Lemma~\ref{lem:wloginvertible} below), the gap is zero and the problems $S(\tilde{\mu}_1, \dots, \tilde{\mu}_N, c)$ and $S(\mu_1, \dots, \mu_N, Kc)$ become equivalent in the following sense.
We write $T_{\sharp}$ for the pushforward under~$T$.

\begin{lemma}
	\label{lem:invertibleequiv}	
	For $i=1, \dots, N$, let $T_i : X_i \rightarrow X_i$ satisfy $\tilde{\mu}_i = ({T_i})_\sharp \mu_i$ and admit a (measurable) a.s.~inverse $T_{i}^{-1}: X_i \rightarrow X_i$; that is, $T_{i}^{-1}\circ T_i= \id$ $\mu_{i}$-a.s.\ and $T_{i}\circ T_{i}^{-1}= \id$ $\tilde{\mu}_{i}$-a.s. Define
	\begin{equation*}
	T(x) = (T_1(x_1), \dots, T_N(x_N)), \quad T^{-1}(x) = (T_{1}^{-1}(x_1), \dots, T_{N}^{-1}(x_N)).
	\end{equation*}
	Then $S(\tilde{\mu}_1, \dots, \tilde{\mu}_N, c)=S(\mu_1, \dots, \mu_N, c\circ T)$ and the optimizers $\tilde{\pi}^{*}$, $\pi^{*}$ of the two problems are related by
	$\tilde{\pi}^{*} = T_\sharp \pi^{*}$ and $\pi^{*} = T^{-1}_\sharp \tilde{\pi}^{*}$.
\end{lemma}

\begin{proof}
		Set $P= \mu_1 \otimes \dots \otimes \mu_N$ and $\tilde{P} = \tilde{\mu}_1 \otimes \dots \otimes \tilde{\mu}_N$. We have
		\begin{align*}
		\int c \circ T \,d\pi + D_f(\pi, P) &= \int c \circ T \,d(T^{-1}_\sharp (T_\sharp\pi)) + D_f(T^{-1}_\sharp (T_\sharp\pi), T^{-1}_\sharp\tilde{P}) \\
		&= \int c \,d(T_\sharp\pi) + D_f(T_\sharp\pi, \tilde{P})
		\end{align*}
		for any $\pi\in\Pi(\mu_1, \dots, \mu_N)$, hence taking infimum over $\pi\in\Pi(\mu_1, \dots, \mu_N)$ yields $S(\mu_1, \dots, \mu_N, c\circ T) \geq S(\tilde{\mu}_1, \dots, \tilde{\mu}_N, c)$. Symmetric results hold starting from $\tilde\pi\in\Pi(\tilde\mu_1, \dots, \tilde\mu_N)$. Thus $S(\tilde{\mu}_1, \dots, \tilde{\mu}_N, c)=S(\mu_1, \dots, \mu_N, c\circ T)$, and now the formulas for the optimizers follow as well.
\end{proof}

In the simplest case, the optimal couplings for $W_{p}(\mu_{i},\tilde{\mu}_{i})$ are given by invertible maps, and then we can apply Lemma~\ref{lem:invertibleequiv} directly to prove Theorem~\ref{th:second}. In general, we approximate the marginals with measures having that property as detailed next, passing to an augmented space to guarantee that the setting is sufficiently rich. We write $\delta_x$ for the Dirac measure at $x$.

\begin{lemma}
	\label{lem:wloginvertible}
	Let $p\in[1,\infty]$. Let $\bar{X}_i = X_i \times (-1, 1)$ and embed the marginals as $\nu_i := \mu_i \otimes \delta_0$ and $\tilde{\nu}_i := \tilde{\mu}_i \otimes \delta_0$ for $i=1, \dots, N$.  Set $\bar{X} = \prod_{i=1}^N \bar{X}_i$ and define $\bar{c} : \bar{X} \rightarrow \mathbb{R}$ by $\bar{c}(x, u) := c(x)$ for $x\in X$ and $u\in (-1, 1)^{N}$.
	\begin{itemize}
	\item[(i)] We have $S(\mu_1, \dots, \mu_N, c)=S(\nu_1, \dots, \nu_N, \bar{c})$ and the corresponding  optimizers $\pi,\theta$ are related by $\theta = \pi \otimes \delta_0^N$.
	
	If $\tilde{\pi}, \tilde\theta$ are the optimizers for $S(\tilde{\mu}_1, \dots, \tilde{\mu}_N, c)$ and $S(\tilde{\nu}_i, \dots, \tilde{\nu}_N, \bar{c})$, then
		\[
		W_p(\pi, \tilde{\pi}) = W_p(\theta, \tilde{\theta}).
		\]
		\item[(ii)] Given $0 < \epsilon < 1$ and $i=1, \dots, N$, there exist $\nu_i^\epsilon, \tilde\nu_i^\epsilon \in \mathcal{P}(\bar{X}_i)$  with
				\begin{equation}
		\label{eq:approxinvertible}
		W_p(\nu_i, \nu_i^\epsilon) \leq \epsilon, \qquad W_p(\tilde{\nu}_i, \tilde{\nu}_i^\epsilon) \leq \epsilon
		\end{equation}
		and an a.s.\ invertible map $T_i^\epsilon: \bar{X}_i \rightarrow \bar{X}_i$ such that $\tilde{\nu}_i^\epsilon = (T_i^\epsilon)_\sharp \nu_i^\epsilon$ and the corresponding coupling attains $W_p(\nu_i^\epsilon, \tilde{\nu}_i^\epsilon)$.
	\end{itemize} 
\end{lemma}
	
\begin{proof}
	(i) follows immediately from the definitions; we prove~(ii). The case $p<\infty$ is standard: for $n$ large enough, there exist $\rho_i, \tilde{\rho}_i \in \mathcal{P}(\bar{X}_i)$ of the form
	\[
	\rho_i = \frac{1}{n} \sum_{k=1}^n \delta_{(x_k, 0)}, \qquad \tilde\rho_i = \frac{1}{n} \sum_{k=1}^n \delta_{(\tilde{x}_k, 0)}
	\]
	such that $W_p(\nu_i, \rho_i) \leq \frac{\epsilon}{2}$ and $W_p(\tilde{\nu}_i, \tilde{\rho}_i) \leq \frac{\epsilon}{2}$; for instance, one can use  suitable realizations of i.i.d.\ samples (e.g., \cite[Corollary~1.1]{Lacker.20}).
	Next, choose distinct $u_1, \dots, u_n \in (0, 1)$ small enough such that the measures 
	\[
	\nu_i^\epsilon = \frac{1}{n} \sum_{k=1}^n \delta_{(x_k, u_k)}, \qquad \tilde\nu_i^\epsilon = \frac{1}{n} \sum_{k=1}^n \delta_{(\tilde{x}_k, u_k)}
	\]
	satisfy $W_p(\rho_i, \nu_i^\epsilon) \leq \frac{\epsilon}{2}$ and $W_p(\tilde{\rho}_i, \tilde{\nu}_i^\epsilon) \leq \frac{\epsilon}{2}$. 
	Then~\eqref{eq:approxinvertible} holds and $\nu_i^\epsilon,\tilde{\nu}_i^\epsilon$ are empirical measures on~$n$ distinct points due to the choice of $u_1, \dots, u_n$. As a result, there is an optimal transport map that is one-to-one on the supports.
	
	Let $p=\infty$. Here a different argument is necessary. (The following also gives an alternate proof for $p<\infty$.) As $X$ is Polish, we can find a dense sequence $(q_k)\subset X$ and a countable measurable partition $(Q_{k})$ of $X$ with $q_{k}\in Q_{k}$ and $\diam Q_{k}\leq \frac{\epsilon}{4}$. Consider the approximations 
	\[
	\rho_i := \sum_{k = 1}^\infty \nu_i(Q_k) \,\delta_{q_k} \otimes \delta_0, \qquad \tilde{\rho}_i := \sum_{k = 1}^\infty \tilde{\nu}_i(Q_k) \,\delta_{q_k} \otimes \delta_0
	\]
	which clearly satisfy $W_\infty(\rho_i, \nu_i) < \frac{\epsilon}{2}$ and $W_\infty(\tilde{\rho}_i, \tilde{\nu}_i) < \frac{\epsilon}{2}$, but may have atoms of unequal mass.
	Let $\rho_i \otimes U_i \in \Pi(\rho_i, \tilde{\rho}_i)$ be a $W_\infty$-optimal coupling, then $U_i : \bar{X}_i \rightarrow \mathcal{P}(\bar{X}_i)$ is a stochastic kernel such that for each $k$,
	\[
	U_i((q_k, 0)) = \sum_{j = 1}^\infty w_{j, k} \, \delta_{q_j} \otimes \delta_0,
	\]
	for some weights $w_{j, k} \geq 0$ with $\sum_{j =1}^\infty w_{j, k} = 1$. Let $\epsilon_0 > 0$ and pick disjoint numbers $u_{j, k} \in (0, \epsilon_{0})$, define
	\[
	\nu_i^{\epsilon} := \sum_{j, k = 1}^\infty \nu_i(Q_k) w_{j, k} \delta_{q_k} \otimes \delta_{u_{j, k}}, \qquad \tilde{\nu}_i^{\epsilon} := \sum_{j, k = 1}^\infty \nu_i(Q_k) w_{j, k} \,\delta_{q_j} \otimes \delta_{u_{j, k}}
	\]
	and observe that $W_\infty(\nu_i^\epsilon, \rho_i) < \frac{\epsilon}{2}$ and  $W_\infty(\tilde\nu_i^\epsilon, \tilde\rho_i) < \frac{\epsilon}{2}$ for $\epsilon_{0}$ sufficiently small (note that $u_{j, k} := 0$ would lead to $\nu_i^\epsilon = \rho_i$ and $\tilde{\nu}_i^\epsilon = \nu_i^\epsilon U_i = \tilde{\rho}_i$). Now \eqref{eq:approxinvertible} holds by the triangle inequality. Define 
	\[
	T_i^\epsilon : \{q_k: k \in \mathbb{N}\} \times \{u_{j, k} : j, k \in \mathbb{N}\} \rightarrow \{q_k: k \in \mathbb{N}\} \times \{u_{j, k} : j, k \in \mathbb{N}\},
	\]
	\[
	T_i^\epsilon(q_k, u_{j, k}) := (q_j, u_{j, k})
	\]
	which is one-to-one as the $u_{j, k}$ are distinct. Moreover, $\rho_i \otimes U_i \in \Pi(\rho_i, \tilde{\rho}_i)$ implies $\tilde{\nu}_i^\epsilon = (T_i^\epsilon)_\sharp \nu_i^\epsilon$, and since $\rho_i \otimes U_i$ attains $W_\infty(\rho_i, \tilde{\rho}_i) = W_\infty(\nu_i^{\epsilon}, \tilde{\nu}_i^{\epsilon})$, the coupling induced by $T_i^\epsilon$ attains $W_\infty(\nu_i^{\epsilon}, \tilde{\nu}_i^{\epsilon})$.
\end{proof}

After these preparations, we are ready to prove Theorem~\ref{th:second}.

\begin{proof}[Proof of Theorem \ref{th:second}]
We detail the proof for \eqref{eq:transportineq}; the argument for~\eqref{eq:transportineq2} is identical.
We shall apply Proposition~\ref{pr:costStability} though the equivalence outlined in Lemma~\ref{lem:invertibleequiv}. To this end, we extend the spaces~$X_{i}$ by the interval $(-1, 1)$ and introduce $\nu_i, \tilde{\nu}_i,\bar{c}$ as in Lemma~\ref{lem:wloginvertible}. In view of Lemma~\ref{lem:wloginvertible}\,(i), it suffices to prove the claim for these data instead of $\mu_i, \tilde{\mu}_i,c$.

	Let $\epsilon > 0$, choose $\nu_i^\epsilon, \tilde{\nu}_i^\epsilon,T^\epsilon$ as in Lemma~\ref{lem:wloginvertible}\,(ii) and denote by $\theta^\epsilon, \tilde{\theta}^\epsilon, \hat{\theta}^\epsilon$ the respective optimizers of $S_{\rm ent}(\nu_1^\epsilon, \dots, \nu_N^\epsilon, \bar{c})$ and $S_{\rm ent}(\tilde\nu_1^\epsilon, \dots, \tilde\nu_N^\epsilon, \bar{c})$ and $S_{\rm ent}(\nu_1^\epsilon, \dots, \nu_N^\epsilon, \bar{c} \circ T^\epsilon)$, respectively. Noting that $\Lip_{p}(\bar{c}) = \Lip_p(c)$ and setting
  $
	\Delta(\epsilon) := W_{p}(\nu_1^\epsilon, \dots, \nu_N^\epsilon;\tilde\nu_1^\epsilon, \dots, \nu_N^\epsilon)
	$, Lemma \ref{lem:costislp} yields 
	\[
	\|\bar{c} - \bar{c} \circ T^\epsilon\|_{L^p(P)} \leq \Lip_p(c) \,\Delta(\epsilon)
	\]
	and thus Proposition~\ref{pr:costStability} shows that
	\begin{align}
	\label{eq:t1ineq_final}
	W_q(\hat{\theta}^\epsilon, \theta^\epsilon) 
	&\leq C_q \left(\frac{1}{2}\right)^{\frac{1}{2q}} \left(a^\frac{1}{p} \Lip_p(c) \,\Delta(\epsilon) \right)^{\frac{p}{(p+1)q}}.
	\end{align}
	As $\tilde{\theta}^\epsilon = {T^\epsilon}_\sharp \hat{\theta}^\epsilon$ by Lemma~\ref{lem:invertibleequiv} and $T_i^\epsilon$ attains $W_p(\nu_i^\epsilon, \tilde{\nu}_i^\epsilon)$, it follows by the same calculation as in the proof of Theorem~\ref{th:opti} that
	\[
	W_q(\tilde{\theta}^\epsilon, \hat{\theta}^\epsilon) \leq N^{(\frac{1}{q} - \frac{1}{p})} W_p(\tilde{\theta}^\epsilon, \hat{\theta}^\epsilon) \leq N^{(\frac{1}{q} - \frac{1}{p})} \Delta(\epsilon).
	\]
	Combining the two estimates, we find that
	\begin{align*}
	W_q(\theta^\epsilon, \tilde{\theta}^\epsilon) &\leq W_q(\theta^\epsilon, \hat{\theta}^\epsilon) + W_q(\hat\theta^\epsilon, \tilde{\theta}^\epsilon)\\ &\leq N^{(\frac{1}{q} - \frac{1}{p})} \Delta(\epsilon) + C_q \left(\frac{1}{2}\right)^{\frac{1}{2q}} \left(a^\frac{1}{p} \Lip_p(c) \,\Delta(\epsilon) \right)^{\frac{p}{(p+1)q}}.
	\end{align*}
	Letting $\epsilon \rightarrow 0$, the left-hand side converges to $W_q(\pi^*, \tilde{\pi}^*)$ by  Theorem~\ref{th:opti} and Lemma~\ref{lem:wloginvertible}\,(ii), while $\Delta(\epsilon)\to \Delta$ by construction. 
	The claim on sharpness is discussed in Example~\ref{ex:sharpness} below.
\end{proof}

Finally, we exhibit a family of examples for which the constant~$\ell$ of Theorem~\ref{th:second} is optimal.

\begin{example}[Sharpness of $\ell$ in Theorem~\ref{th:second}.]
	\label{ex:sharpness}
	On $X = [-1,1]^{2}$, let
	\begin{align*}
	\mu_1 = \mu_2 = \frac{1}{2} \left( \delta_{-1} + \delta_1 \right), \qquad \tilde{\mu}_1 = \tilde{\mu}_2 = \frac{1}{2} \left( \delta_{-1+\eps} + \delta_{1 - \eps} \right),
	\end{align*}
	where $\eps\in(0,1/2)$ is a parameter. We define the cost function $c=c(\eps)$ by
	\begin{align*}
	c(-1, -1)& &=& &c(1, 1) &&=& &c(-1+\eps, 1-\eps) &&=& &c(1-\eps, -1+\eps) &&= 0, \\
	c(1, -1) &&=& &c(-1, 1) &&=& &c(-1+\eps, -1+\eps)& &=& &c(1-\eps, 1-\eps) &&= \eps,
	\end{align*}
  then $c$ is Lipschitz with constant $\Lip_{\infty}(c)=1$. Setting $\alpha(\eps) := \frac{\exp(\eps)}{1+\exp(\eps)}$, we calculate the optimizers $\pi^*,\tilde{\pi}^*$ of $S_{\textrm{ent}}(\mu_1,\mu_2, c)$ and $S_{\textrm{ent}}(\tilde{\mu}_1,\tilde{\mu}_2, c)$ 
   to be
	\begin{align*}
	\pi^* &= \frac{\alpha(\epsilon)}{2} \big(\delta_{(-1, -1)} + \delta_{(1, 1)}\big) + \frac{1-\alpha(\epsilon)}{2} \big(\delta_{(-1, 1)} + \delta_{(1, -1)}\big),  \\
	\tilde{\pi}^* &=  \frac{1-\alpha(\epsilon)}{2} \big(\delta_{(-1+\eps, -1+\eps)} + \delta_{(1-\eps, 1-\eps)}\big) + \frac{\alpha(\epsilon)}{2} \big(\delta_{(1-\eps, -1+\eps)} + \delta_{(-1+\eps, 1-\eps)}\big).
	\end{align*}
	Next, we find
	\[
	  W_1(\pi^*, \tilde{\pi}^*) = 2 ( 1- \alpha(\eps)) 2 \eps + (2\alpha(\eps) - 1)2
	\]
	by observing that an optimal coupling $\kappa\in\Pi(\pi^*,\tilde{\pi}^*)$ is to move a total mass of $2 ( 1- \alpha(\eps))$ over a $d_{X,1}$-distance of $2\eps$ and mass $2\alpha(\eps) - 1$ over distance $(2-\eps) + \eps=2$. In view of $\alpha(\eps) = \frac{1}{2} + \frac{\eps}{4} + \mathcal{O}(\eps^3)$ as $\eps\to0$, we deduce
	\[
	  W_1(\pi^*, \tilde{\pi}^*) = 3 \eps + \mathcal{O}(\eps^2).
	\]
  On the other hand, clearly
	\[
	   W_\infty(\mu_1,\mu_2;\tilde\mu_1,\tilde\mu_2) = \eps.
	\]
	In summary, any constant $\ell$ such that $W_1(\pi^*, \tilde{\pi}^*)\leq \ell W_\infty(\mu_1,\mu_2;\tilde\mu_1,\tilde\mu_2)$ holds in the above example for all~$\eps$, has to satisfy $\ell\geq3$.
	
	It remains to see that we attain $\ell=3$ in the last assertion of Theorem~\ref{th:second}. For $q=1$, Lemma~\ref{le:transportineq}\,(i) with $\diam_{1}(X_{2})=\diam ([-1,1])=2$ shows that~\eqref{eq:transportineq} is satisfied with $C_1=\sqrt2$. Hence, the formula in  Theorem~\ref{th:second} reads
	$$
	  \ell= N + (C_1/\sqrt2)\Lip_{\infty}(c)= 2+ 1 =3
	$$
	as desired.
	
	We remark that this example can be extended to more general parameters. Replacing~$c$ by~$L  c$ for some $L > 0$ leads to a different Lipschitz constant in the definition of~$l$. Replacing $\alpha(\varepsilon)$ by $\alpha(L\varepsilon)$ in the formula for $W_1(\pi^*, \tilde{\pi}^*)$, one finds that the constant~$l$ is again sharp. Similarly, replacing $[-1, 1]$ by $[-K, K]$ for some $K>0$ and replacing $1$ by $K$ in the definition of the marginals, we find that only the constant $C_1$ changes in the definition of~$l$, while for $W_1(\pi^*, \tilde{\pi}^*)$ one replaces the final $2$ by $2K$. Again, the constant~$l$ remains sharp.
\end{example}
	
\subsection{Application to Sinkhorn's Algorithm}

\begin{proof}[Proof of Theorem \ref{th:sinkhorn}]
  We first observe that $\pi^{n}$ is the optimizer of the problem $S_{\rm ent}(\pi^{n}_1, \pi^{n}_2, c)$:
		\begin{align*}
		\pi^{n} &= \argmin_{\pi \in \Pi(\pi^{n}_1, \pi^{n}_2)} D_{\rm KL}(\pi, \pi^{0}) \\
		&= \argmin_{\pi \in \Pi(\pi^{n}_1, \pi^{n}_2)} \int c \,d\pi + D_{\rm KL}(\pi, \mu_{1}\otimes\mu_{2}) \\
		&= \argmin_{\pi \in \Pi(\pi^{n}_1, \pi^{n}_2)} \int c \,d\pi + D_{\rm KL}(\pi, \pi^{n}_1\otimes \pi^{n}_2),
		\end{align*}  
  where the last step used Remark~\ref{rk:otherRef}. (The first identity is well known; e.g., it follows from the fact that by construction, $d\pi^{n}/d\pi^{0}$ admits a factorization $a(x_{1})b(x_{2})$.) To apply our stability results, we require the convergence of the marginals in~$W_{p}$. Indeed, $D_{\rm KL}(\pi^{n}_i, \mu_i)\to0$ holds by a standard entropy calculation, see for instance~\cite{Ruschendorf.95}. More precisely, we have
		\begin{equation}\label{eq:Leger}
		D_{\rm KL}(\pi^{n}_i, \mu_i) \leq 2 \frac{D_{\rm KL}(\pi^*, P_c)}{n}
		\end{equation}
		according to~\cite[Corollary~1]{Leger.21}.		
		By the exponential moment condition on $\mu_i$ and \cite[Corollary~2.3]{BolleyVillani.05}, \eqref{eq:Leger} yields
		\[
		W_p(\pi^n_i, \mu_i) \leq C_{0} C_{\mu_{i}} (n^{-\frac{1}{p}} + n^{-\frac{1}{2p}}) \qquad\mbox{where}
		\]
		\begin{align*}
		C_0 &:=\max\left\{\left(2 D_{\rm KL}(\pi^*, P_c)\right)^{\frac{1}{p}}, \left(2 D_{\rm KL}(\pi^*, P_c)\right)^{\frac{1}{2p}}\right\},\\
		C_{\mu_i} &:= 2 \inf_{x_0 \in X_i, \alpha > 0} \left( \frac{1}{\alpha} \Big(\frac{3}{2} + \log \int \exp(\alpha d_{X_i}(x_0, x_i)) \,\mu_i(dx_i)\Big)\right)^{\frac{1}{p}}.
		\end{align*}
  As a result,
		\begin{equation}\label{eq:SinkDelta}
		  \Delta := \max_{i=1,2}W_p(\pi^{n}_i, \mu_i) \leq C_{0} \max\{C_{\mu_{1}},C_{\mu_{2}}\} (k^{-\frac{1}{p}} + k^{-\frac{1}{2p}}).
		\end{equation}
		We remark that $\Delta=W_{p}(\pi^{n}_1,\pi^{n}_2;\mu_1,\mu_2)$ as $W_p(\pi^{n}_1, \mu_1)=0$ or $W_p(\pi^{n}_2, \mu_2)=0$ for each~$n$, consistent with our previous notation. 
		By~\eqref{eq:SinkDelta}, $\pi^{n}_i$ has a finite $p$-th moment. Noting also that
		\begin{equation}\label{eq:ref_split}D_{\rm KL}(\pi^{n}, \mu_{1}\otimes\mu_{2}) = D_{\rm KL}(\pi^{n}, \pi^{n}_1 \otimes \pi^{n}_2) + \sum_{i=1}^2 D_{\rm KL}(\pi^{n}_i, \mu_i),\end{equation}
		assertion~(i) thus follows directly from Theorem~\ref{th:valueConv}\,(i). 
		
		Regarding~(ii), note that the $p$-th moments of $\pi^{n}_i$ are bounded uniformly in~$n$ due to~\eqref{eq:SinkDelta}. In view of Lemma~\ref{le:ccond}, the cost function~$c$ thus satisfies~\eqref{eq:ccond} with a uniform constant~$L$ for the marginals $(\pi^{n}_1, \pi^{n}_2)_{n}$ as well as $(\mu_1, \mu_2)$.
		Using also~\eqref{eq:Leger} and~\eqref{eq:ref_split}, Theorem~\ref{th:valueConv}\,(ii)  yields 
		$$
		| \mathcal{F}(\pi^*) - \mathcal{F}(\pi^{n}) | \leq L\,\Delta + 2 D_{\rm KL}(\pi^*, P_c) n^{-1}.
		$$
    In view of~\eqref{eq:SinkDelta}, the claimed rate for $| \mathcal{F}(\pi^*) - \mathcal{F}(\pi^{n}) |$ follows. 
		Finally, $({\rm I}_{q}^{'})$ holds with constant $C_q^{'}$ by Lemma~\ref{le:transportineq}\,(iii) and thus Theorem~\ref{th:opti} yields 
		$$
		W_q(\pi^*, \pi^{n}) \leq 2^{(\frac{1}{q} - \frac{1}{p})} \, \Delta + C_q^{'} (2L)^{1/q} \,\Delta^{\frac{1}{q}} + C_q^{'} L^{\frac{1}{2q}} \,\Delta^{\frac{1}{2q}},
		$$
		so that the claimed rate for $W_q(\pi^*, \pi^{n})$ follows via~\eqref{eq:SinkDelta}.
\end{proof}

\newcommand{\dummy}[1]{}

\end{document}